\documentclass[12pt]{amsart}

\def\bR {\mathbf{R}}
\def\bS {\mathbf{S}}
\def\bT {\mathbf{T}}

\def\cB {\mathcal{B}}

\def\cD {\mathcal{D}}

\def\cG {\mathcal{G}}
\def\cH {\mathcal{H}}

\def\cL {\mathcal{L}}
\def\cM {\mathcal{M}}

\def\cP {\mathcal{P}}
\def\cQ {\mathcal{Q}}
\def\cR {\mathcal{R}}

\def\a {{\alpha}}

\def\G {{\Gamma}}
\def\de {{\delta}}
\def\eps {{\epsilon}}
\def\th {{\theta}}

\def\l {{\lambda}}
\def\L {{\Lambda}}
\def\si {{\sigma}}
\def\Si {{\Sigma}}

\def\om {{\omega}}
\def\Om {{\Omega}}

\def\d {{\partial}}
\def\grad {{\nabla}}
\def\Dlt {{\Delta}}

\def\rstr {{\big |}}
\def\indc {{\bf 1}}

\def\la {\langle}
\def\ra {\rangle}

\def\Sh {{\hbox{S\!h}}}

\def\Ma{{\hbox{M\!a}}}
\def\Re{{\hbox{R\!e}}}

\newcommand{\Div}{\operatorname{div}}
\newcommand{\Rot}{\operatorname{curl}}

\newcommand{\Span}{\operatorname{span}}
\newcommand{\Supp}{\operatorname{supp}}

\newcommand{\Ker}{\operatorname{Ker}}

\newcommand{\ba}{\begin{aligned}}
\newcommand{\ea}{\end{aligned}}

\newcommand{\be}{\begin{equation}}
\newcommand{\ee}{\end{equation}}

\newcommand{\lb}{\label}

\newtheorem{Thm}{Theorem}[section]

\newtheorem{Prop}[Thm]{Proposition}
\newtheorem{Cor}[Thm]{Corollary}
\newtheorem{Lem}[Thm]{Lemma}
\newtheorem{Def}[Thm]{Definition}

\begin{document}

\title[Euler Limit of Boltzmann with Accommodation]{The Incompressible Euler Limit\\ of the Boltzmann Equation
\\ with Accommodation Boundary Condition}

\author[C. Bardos]{Claude Bardos}
\address[C.B. \& F.G.]{Universit\'e Paris-Diderot, Laboratoire J.-L. Lions, BP 187, 75252 Paris Cedex 05, France}
\email{claude.bardos@gmail.com}

\author[F. Golse]{Fran\c cois Golse}
\address[F.G.]{Ecole Polytechnique, Centre de Math\'ematiques L. Schwartz, 91128 Palaiseau Cedex France}
\email{francois.golse@math.polytechnique.fr}

\author[L. Paillard]{Lionel Paillard}
\address[L.P.]{Ecole Polytechnique, Centre de Math\'ematiques L. Schwartz, 91128 Palaiseau Cedex France}
\email{lionel.paillard@club-internet.fr}

\begin{abstract}
The convergence of solutions of the incompressible Navier-Stokes equations set in a domain with boundary to solutions of the Euler equations in the large Reynolds 
number limit is a challenging open problem both in 2 and 3 space dimensions. In particular it is distinct from the question of existence in the large 
of a smooth solution of the initial-boundary value problem for the Euler equations. The present paper proposes three results in that direction. First, if 
the solutions of the Navier-Stokes equations satisfy a slip boundary condition with vanishing slip coefficient in the large Reynolds number limit, we 
show by an energy method that they converge to  the classical solution of the Euler equations on its time interval of existence. Next we show that the 
incompressible Navier-Stokes limit of the Boltzmann equation with Maxwell's accommodation condition at the boundary is governed by the 
Navier-Stokes equations with slip boundary condition, and we express the slip coefficient at the fluid level in terms of the accommodation parameter 
at the kinetic level. This second result is formal, in the style of [Bardos-Golse-Levermore, J. Stat. Phys. \textbf{63} (1991), 323--344]. Finally, we 
establish the incompressible Euler limit of the Boltzmann equation set in a domain with boundary with Maxwell's accommodation condition assuming 
that the accommodation parameter is small enough in terms of the Knudsen number. Our proof uses the relative entropy method following closely 
[L. Saint-Raymond, Arch. Ration. Mech. Anal. {\bf 166} (2003), 47--80] in the case of the 3-torus, except for the boundary terms, which require special treatment.
\end{abstract}

\keywords{Navier-Stokes equations; Euler equations; Boltzmann equation; Fluid dynamic limit; Inviscid limit; Slip coefficient; Maxwell's accommodation
boundary condition; Accomodation parameter; Relative entropy method; Dissipative solutions of the Euler equations}

\subjclass{35Q30, 82B40; (76D05, 76B99)}

\maketitle

\rightline{\textit{To C. David Levermore}}

\section{Introduction}

In a program initiated more than 20 years ago with Dave Levermore \cite{BGL0,BGL1,BGL2}, we outlined a strategy for deriving incompressible 
fluid dynamic equations from the theory of renormalized solutions of the Boltzmann equation invented by R. DiPerna and P.-L. Lions \cite{dPL}. 

At the time of this writing, complete derivations of the Stokes \cite{LionsMas1,LionsMas2}, Stokes-Fourier \cite{GL} and Navier-Stokes-Fourier
\cite{GSR1,GSR2,LevMas} have been obtained following that program, in the greatest possible generality allowed by the current existence 
theories for both the fluid dynamic and the Boltzmann equations: see \cite{VilBbki} for a survey of these issues.

The case of the incompressible Euler equations in space dimension 3 stands out, in the first place because there does not exist a satisfactory 
theory of global weak solutions of these equations analogous to Leray's theory of weak solutions of the Navier-Stokes equations \cite{Leray34}
in space dimension 3. Even if there was a global existence theory of weak solutions of the incompressible Euler equations in the energy space 
$L^\infty_t(L^2_x)$ in dimension 3, such solutions would not satisfy the weak-strong uniqueness property observed by Leray in the case of the
Navier-Stokes equations. (Indeed there exist nontrivial compactly supported solutions of the incompressible Euler equations in energy space:
see \cite{Sheffer,Shnirel,deLelSzeke}.) In \cite{LionsBk1}, P.-L. Lions proposed a notion of dissipative solution of the incompressible Euler
equations --- in the same spirit of his definition of the notion of viscosity solutions of Hamilton-Jacobi equations, but using the conservation 
of energy instead of the maximum principle as in the Hamilton-Jacobi case. The weak-strong uniqueness property is verified by dissipative
solutions of the incompressible Euler equations (essentially by definition): if there exists a classical ($C^1$) solution of the incompressible
Euler equations, all dissipative solutions with the same initial data must coincide with this classical solution on its maximal time interval of
existence. Unfortunately, dissipative solutions are not known to satisfy the incompressible Euler equations in the sense of distributions.

Using the relative entropy method pioneered in \cite{HTYau} and adapted to the case of the Boltzmann equation in \cite{BGP,LionsMas2},
L. Saint-Raymond \cite{LSR1,LSR2} succeeded in deriving dissipative solutions of the Euler equations in arbitrary space dimension (or 
classical solutions whenever they exist) from weak solutions of the BGK model \cite{LSR1} or from renormalized solutions of the Boltzmann 
equation \cite{LSR2}.

However, all the derivations of fluid dynamic equations from the Boltzmann equation referred to above are carried out in either the Euclidian
 space $\bR^N$ of the flat torus $\bT^N$ so as to avoid difficulties that may result from boundary conditions. The theory of renormalized
 solutions of the Boltzmann equation in the presence of accommodation boundary conditions was obtained only very recently, by S. Mischler
 \cite{Misch}; subsequently, N. Masmoudi and L. Saint-Raymond established the Stokes-Fourier limit of such solutions \cite{MasSR}.
 
 In the present paper, we derive dissipative solutions (or classical solutions whenever they exist) of the incompressible Euler equations from
 renormalized solutions of the Boltzmann equation in some spatial domain satisfying Maxwell's accommodation boundary condition. In particular,
 we identify a sufficient scaling condition on the accommodation parameter under which the hydrodynamic limit of the family of solutions of the
 Boltzmann equation is governed by the incompressible Euler equation with its classical boundary condition --- i.e. assuming that the velocity
 field is tangent at the boundary.
 
 The outline of the paper is as follows. Section \ref{S-NSEul} gives a sufficient condition on the slip coefficient at the boundary under which the 
 incompressible Euler equations are obtained as the inviscid limit of the incompressible Navier-Stokes equations with slip-boundary condition.
 The main result in this section is Theorem \ref{T-Eul<NS}, based on an energy method. Section \ref{S-BNS} provides a formal derivation of the
 incompressible Navier-Stokes equations with slip boundary condition from the Boltzmann equation with Maxwell's accommodation condition at 
 the boundary of the spatial domain: see Theorem \ref{T-FormNSLim} for a precise statement of this result. Based on the intuition provided by
 sections \ref{S-NSEul} and \ref{S-BNS}, we identify a scaling limit of the Boltzmann equation with Maxwell accommodation boundary condition
 leading to the incompressible Euler equations: see Theorem \ref{T-Eul<B}, whose proof occupies most of section \ref{S-BEul}. 
 
 It is a our great pleasure to offer this modest contribution to our friend Dave Levermore, in recognition of his outstanding influence on the
 analysis of nonlinear partial differential equations in the past 30 years, especially on the problem of hydrodynamic limits of the Boltzmann 
 equation, directly inspired from Hilbert's 6th problem on the axiomatization of physics.
 
\section{Inviscid Limit of the Navier-Stokes Equations\\ with Slip Boundary Conditions}\lb{S-NSEul}

A a warm-up, we begin with a simple observation bearing on the inviscid limit of the incompressible Navier-Stokes equations set in some
smooth domain with  slip boundary condition. In particular, we identify a sufficient scaling condition on the slip coefficient in order to obtain
the incompressible Euler equations in the inviscid limit.

Let $\Om$ designate an open set in $\bR^N$ with $C^1$ boundary $\d\Om$, assuming that $N=2$ or $3$; henceforth the outward unit normal 
vector at the point $x$ of $\d\Om$ is denoted by $n_x$. Consider the initial-boundary value problem with unknown $u_\nu=u_\nu(t,x)$, set for 
$x\in\Om$ and $t\ge 0$:
\be\lb{NSSlip}
\left\{
\ba
{}&\Div_xu_\nu=0\,,
\\
&\d_tu_\nu+\Div_x(u_\nu\otimes u_\nu)+\grad_xp_\nu=\nu\Dlt_xu_\nu\,,
\\
&u_\nu\cdot n\rstr_{\d\Om}=0\,,
\\
&\nu(\Si(u_\nu)\cdot n)_\tau+\l u_\nu\rstr_{\d\Om}=0\,,
\\
&u_\nu\rstr_{t=0}=u^{in}\,,
\ea
\right.
\ee
where $\nu>0$ is the kinematic viscosity, $\l>0$ the slip coefficient, $\Si(u):=\grad_xu+(\grad_xu)^T$, while 
$$
v(x)_\tau:=(I-n(x)^{\otimes 2})\cdot v(x)\,.
$$

Henceforth, we denote
$$
\cH(\Om):=\{v\in L^2(\Om;\bR^N)\,|\,\Div v=0\hbox{ and }v\cdot n\rstr_{\d\Om}=0\}\,.
$$

For each $\nu>0$ and $u^{in}\in\cH(\Om)$, there exists a weak solution $u_\nu$ of (\ref{NSSlip}) in 
$L^\infty(\bR_+;\cH(\Om))\cap L^2(\bR_+;H^1(\Om))$, meaning that, for each test vector field
$U\in C(\bR_+;\cH(\Om))\cap C^\infty_c(\bR_+\times\overline{\Om})$, one has
\be\lb{WeakNS}
\ba
\nu&\int_0^\infty\int_\Om\tfrac12\Si(u_\nu):\Si(U)dxdt+\l\int_0^\infty\int_{\d\Om}u_\nu\cdot UdS(x)dt
\\
&=
\int_0^\infty\int_\Om(u_\nu\cdot\d_tU+u_\nu\otimes u_\nu:\grad_xU)dxdt
+
\int_\Om u^{in}(x)\cdot U(0,x)dx\,,
\ea
\ee
and satisfying in addition $u_\nu\in C(\bR_+;w-L^2(\Om))$, together with the Leray-type energy dissipation inequality:
\be\lb{Leray<}
\ba
\int_\Om\tfrac12|u_\nu(t,x)|^2dx&+\nu\int_0^t\int_\Om|\Si(u_\nu)(t,x)|^2dxdt
\\
&+\l\int_0^t\int_{\d\Om}|u_\nu(t,x)|^2dS(x)dt\le\int_\Om\tfrac12|u^{in}(x)|^2dx
\ea
\ee
for each $t\ge 0$. Such a weak solution of (\ref{NSSlip}) will henceforth be referred to as a ``Leray solution of (\ref{NSSlip})''. The classical theory
of Leray solutions that is well known in the case where the velocity field satisfies the Dirichlet boundary condition on $\d\Om$ can be adapted to
the case of the slip-boundary condition: see \cite{Solo, Beirao}, and Theorem 2 in \cite{IftiSueur}.

Since $u_\nu\in C(\bR_+;w-L^2(\Om))$, by an elementary density argument, one can choose a sequence of test vector fields $U_n$ of the special 
form $U_n(t,x)=\chi_n(T-t)w(t,x)$ where $w\in C(\bR_+;\cH(\Om))\cap C^1_c(\bR_+\times\overline\Om)$ and 
$$
\chi_n(z)=\int_{-\infty}^z\chi_n(s)ds
$$
where $\chi_n'$ is a regularizing sequence on $\bR$, so that the weak formulation (\ref{WeakNS}) of the Navier-Stokes equations becomes
$$
\ba
\nu\int_0^T\int_\Om\tfrac12\Si(u_\nu):\Si(w)dxdt+\l\int_0^T\int_{\d\Om}u_\nu\cdot wdS(x)dt
\\
=\int_\Om u^{in}(x)\cdot w(0,x)dx-\int_\Om u_\nu(T,x)w(T,x)dx
\\
+\int_0^T\int_\Om(u_\nu\cdot\d_tw+u_\nu\otimes u_\nu:\grad_xw)dxdt\,,
\ea
$$
for each $T>0$. Furthermore, denoting
$$
E(w):=\d_tw+w\cdot\grad_xw\,,
$$
one has
$$
\int_0^T\int_\Om u_\nu\cdot\d_twdxdt=\int_0^T\int_\Om (u_\nu\cdot E(w)-u_\nu\otimes w\cdot\grad_xw)dxdt
$$
while
\be\lb{EnergEulTest}
\ba
\int_0^T\tfrac12|w(t,x)|^2dx-\int_0^T\tfrac12|w(0,x)|^2dx+\int_0^T\int_\Om w\otimes w:\grad_xwdxdt
\\
=\int_0^T\int_\Om w\cdot E(w)dxdt\,.
\ea
\ee
Therefore
\be\lb{WSol2}
\ba
\nu\int_0^T\int_\Om\tfrac12\Si(u_\nu):\Si(w)dxdt+\l\int_0^T\int_{\d\Om}u_\nu\cdot wdS(x)dt
\\
=\int_0^T\int_\Om(u_\nu\cdot E(w)+u_\nu\otimes(u_\nu-w):\grad_xw-w\otimes u_\nu:\grad_xw)dxdt
\\
+\int_\Om u^{in}(x)\cdot U(0,x)dx-\int_\Om u_\nu(T,x)w(T,x)dx\,,
\ea
\ee
since
$$
\int_\Om w\otimes u_\nu:\grad_xwdx=\int_\Om\Div_x(u_\nu\tfrac12|w|^2)dx=0\,,
$$
because $u_\nu\cdot n\rstr_{\d\Om}=0$. Combining (\ref{WSol2}) and (\ref{EnergEulTest}), we find that
\be\lb{WSol3}
\ba
\nu\int_0^T\int_\Om\tfrac12\Si(u_\nu):\Si(w)dxdt+\l\int_0^T\int_{\d\Om}u_\nu\cdot wdS(x)dt
\\
=\int_0^T\int_\Om((u_\nu-w)\cdot E(w)+(u_\nu-w)\otimes(u_\nu-w):\grad_xw)dxdt
\\
+\int_\Om u^{in}(x)\cdot w(0,x)dx-\int_\Om\tfrac12|w(0,x)|^2dx
\\
-\int_\Om u_\nu(T,x)w(T,x)dx+\int_\Om\tfrac12|w(T,x)|^2dx
\ea
\ee
Finally, combining (\ref{Leray<}) and (\ref{WSol3}), we conclude that any Leray solution $u_\nu$ of (\ref{NSSlip}) satisfies the inequality
\be\lb{RelEnergNS<}
\ba
\int_\Om\tfrac12|u_\nu-w|^2(t,x)dx+\int_0^t\int_\Om(u_\nu-w)\otimes(u_\nu-w):\grad_xwdxds
\\
+\nu\int_0^t\int_\Om\tfrac12|\Si(u_\nu)(s,x)|^2dxds+\l\int_0^t\int_{\d\Om}|u_\nu(s,x)|^2dS(x)ds
\\
\le\int_\Om\tfrac12|u^{in}(x)-w(0,x)|^2dx+\int_0^t\int_\Om E(w)\cdot(u_\nu-w)dxdt
\\
+\nu\int_0^t\int_\Om\tfrac12\Si(u_\nu):\Si(w)dxds+\l\int_0^t\int_{\d\Om}u_\nu\cdot wdS(x)ds
\ea
\ee
for each $w\in C(\bR_+;\cH(\Om))\cap C^1_c(\bR_+\times\overline\Om)$.

At this point we recall the definition of dissipative solutions of the incompressible Euler equations set in a domain $\Om$ with smooth boundary:
\be\lb{IncEul}
\left\{
\ba
{}&\Div_xu=0\,,
\\
&\d_tu+\Div_x(u\otimes u)+\grad_xp=0\,,
\\
&u\cdot n\rstr_{\d\Om}=0\,,
\\
&u\rstr_{t=0}=u^{in}\,.
\ea
\right.
\ee

\begin{Def}[P.-L. Lions \cite{LionsBk1}, p. 154, C. Bardos, E. Titi \cite{BardTiti}, p. 16]
Given $u^{in}\in\cH(\Om)$, a dissipative solution of (\ref{IncEul}) is an element $u\in C(\bR_+;w-\cH(\Om))$ satisfying $u\rstr_{t=0}=u^{in}$
and the inequality
\be\lb{DefDissSol}
\ba
\int_\Om\tfrac12&|u-w|^2(t,x)dx
\\
&\le\exp\left(\int_0^t2\|\si(w)^-\|_{L^\infty(\Om)}(s)ds\right)\int_\Om\tfrac12|u^{in}(x)-w(0,x)|^2dx
\\
&+\int_0^t\exp\left(\int_s^t2\|\si(w)^-\|_{L^\infty(\Om)}(\tau)d\tau\right)\int_\Om E(w)\cdot(u-w)(s,x)dxds
\ea
$$
for each $w\in C(\bR_+;\cH(\Om))\cap C^1(\bR_+\times\overline\Om)$, where
$$
\si(w)^-(t,x):=\sup_{|\xi|=1}(-\Si(w)(t,x):\xi\otimes\xi)\,.
\ee
\end{Def}

We recall that, if the Euler equations (\ref{IncEul}) have a classical solution $v\in C^1([0,T^*)\times\overline\Om)$ satisfying
$$
\si(v)^-\in L^1([0,T];L^\infty(\Om))\hbox{ and }p\in L^1([0,T];H^1(\Om))\hbox{ for each }T<T^*\,,
$$
then all dissipative solutions of (\ref{IncEul}) must coincide with $v$ on $[0,T^*)\times\Om$ a.e., since one can use $w=v$ as the test vector 
field, so that
$$
\int_\Om E(v)\cdot(u-v)(s,x)dx=-\int_\Om\grad_xp\cdot(u-v)(s,x)dx=0
$$
because $(u-v)(s,\cdot)\in\cH(\Om)$ for each $s\in[0,T)$.

\begin{Thm}\lb{T-Eul<NS}
Let $u^{in}\in\cH(\Om)$, and assume that the slip coefficient $\l\equiv\l(\nu)$ in (\ref{NSSlip}) scales with the kinematic viscosity $\nu$ so that
\be\lb{ScalSlip}
\l(\nu)\to 0\hbox{ as }\nu\to 0\,.
\ee
Then any family $(u_\nu)$ of Leray solutions of (\ref{NSSlip}) is relatively compact in the weak-* topology of $L^\infty(\bR_+;\cH(\Om))$ and
in $C(\bR_+;w-\cH(\Om))$ for the topology of uniform convergence on bounded time intervals, and each limit point of $(u_\nu)$ as $\nu\to 0$
is a dissipative solution of (\ref{IncEul}).
\end{Thm}

\begin{proof}
We deduce from (\ref{RelEnergNS<}) with $w=0$, or equivalently from the Leray energy inequality that
\be\lb{UnifBndsNS}
\ba
\sqrt\nu\Si(u_\nu)\hbox{ is bounded in }L^2(\bR_+;L^2(\Om))\,,\hbox{ and }&
\\
\sqrt{\l(\nu)}u_\nu\rstr_{\d\Om}\hbox{ is bounded in }L^2(\bR_+;L^2(\d\Om))&\,.
\ea
\ee
By Gronwall's inequality
\be\lb{DissipIneq-nu}
\ba
\int_\Om&\tfrac12|u_\nu-w|^2(t,x)dx
\\
&\le\exp\left(\int_0^t2\|\si(w)^-\|_{L^\infty(\Om)}(s)ds\right)\int_\Om\tfrac12|u^{in}(x)-w(0,x)|^2dx
\\
&+\int_0^t\exp\left(\int_s^t2\|\si(w)^-\|_{L^\infty(\Om)}(\tau)d\tau\right)\int_\Om E(w)\cdot(u_\nu-w)(s,x)dxds
\\
&+\int_0^t\exp\left(\int_s^t2\|\si(w)^-\|_{L^\infty(\Om)}(\tau)d\tau\right)Q_\nu(s)ds
\ea
\ee
where, by the Cauchy-Schwarz inequality 
$$
\ba
Q_\nu(s)&=\nu\|\Si(u_\nu)\|_{L^2(\Om)}(s)\|\Si(w)\|_{L^2(\Om)}(s)
\\
&+\l(\nu)\|u_\nu\|_{L^2(\d\Om)}(s)\|w\|_{L^2(\Om)}(s)\,.
\ea
$$
In view of (\ref{UnifBndsNS}), one has
$$
\|Q_\nu\|_{L^1([0,T])}=O(\sqrt{\nu})+O(\sqrt{\l(\nu)})\to 0
$$
as $\nu\to 0$, and we conclude by passing to the limit in (\ref{DissipIneq-nu}) following the same argument as in \cite{LionsBk1}.
\end{proof}

Several remarks are in order after this result.

In some references, the slip boundary condition is written
\be\lb{SlipNormDer}
\nu\left(\frac{\d u_\nu}{\d n}\right)_\tau+\l u\rstr_{\d\Om}=0\,,
\ee
instead of 
\be\lb{SlipFour}
\nu(\Si(u_\nu)\cdot n)_\tau+\l u\rstr_{\d\Om}=0\,.
\ee
Likewise, the boundary condition
$$
\Rot u_\nu\times n\rstr_{\d\Om}=0
$$
is also considered by some authors --- and referred to as the Navier slip condition --- in the context of the inviscid limit of the Navier-Stokes equations: 
see for instance \cite{JLLNonLin, BardosEuler, Beirao, BeiraoCrispo2}.

If $\d\Om$ is a straight line, or a plane, or a hyperplane in space dimension $N>3$, the normal vector field $n$ is constant, so that
$$
(\Si(u_\nu)\cdot n)_\tau\rstr_{\d\Om}=\left(\frac{\d u_\nu}{\d n}\right)_\tau\rstr_{\d\Om}+\grad^\tau(u\cdot n\rstr_{\d\Om})
	=\left(\frac{\d u_\nu}{\d n}\right)_\tau\rstr_{\d\Om}
$$
and
$$
\ba
\Rot u_\nu\times n\rstr_{\d\Om}&=(\grad_xu_\nu-(\grad_xu)^T)\cdot n\rstr_{\d\Om}
\\
&=\left(\frac{\d u_\nu}{\d n}\right)_\tau\rstr_{\d\Om}-\grad^t(u\cdot n\rstr_{\d\Om})=\left(\frac{\d u_\nu}{\d n}\right)_\tau\rstr_{\d\Om}
\ea
$$
(with $\grad^\tau$ denoting the tangential component of the $\grad$ operator), since the velocity field $u_\nu$ is tangential on $\d\Om$. Therefore, in
the case of a flat boundary, all these boundary conditions are equivalent. 

If $\d\Om$ is a smooth curve, or a surface or a hypersurface in space dimension $N>3$, then
$$
(\Si(u_\nu)\cdot n)_\tau\rstr_{\d\Om}=\left(\frac{\d u_\nu}{\d n}\right)_\tau-\grad^tn\cdot u_\tau\rstr_{\d\Om}
$$
while
$$
\Rot u_\nu\times n\rstr_{\d\Om}=\left(\frac{\d u_\nu}{\d n}\right)_\tau+\grad^\tau n\cdot u_\tau\rstr_{\d\Om}\,,
$$
so that all these boundary conditions differ by a 0-order operator given by the Weingarten endomorphism of the boundary $\d\Om$. 

Here, we have chosen the second boundary condition above, as it is the more natural one when looking at the Navier-Stokes equation as a fluid 
dynamic limit of the kinetic theory of gases.

However, the same argument as in the proof of Theorem \ref{T-Eul<NS} can be extended to treat the case of a slip coefficient $\l$ which is not
nonnegative, provided that
$$
\l(\nu)_+=\max(\l(\nu),0)\to 0\quad\hbox{ and }\l_-(\nu)=\max(-\l(\nu),0)=O(\nu)
$$
as $\nu\to 0$. Indeed, the contribution of $\l(\nu)_-$ in the estimate (\ref{RelEnergNS<}) can be absorbed in the viscous dissipation term by means
of the following classical inequality: for each $\a>0$, there exists $C_\a>0$ such that, for each $v\in H^1(\Om)$,
$$
\int_{\d\Om}|v(x)|^2dS(x)\le\a\int_\Om|\grad_xv(x)|^2dx+\frac{C_\a}{\a}\int_\Om|v(x)|^2dx\,.
$$
With this observation, the term 
$$
(\Si(u_\nu)\cdot n)_\tau\rstr_{\d\Om}
$$ 
can be replaced indifferently with either
$$
\left(\frac{\d u_\nu}{\d n}\right)_\tau\rstr_{\d\Om}\quad\hbox{ or }\quad\Rot u_\nu\times n\rstr_{\d\Om}
$$
in the slip boundary condition.

More precise variants of Theorem \ref{T-Eul<NS} have been established by various authors, see for instance \cite{XiaoXin,BeiraoCrispo1}.
The result given here holds for a very general class of nonnegative slip boundary coefficients $\l$ and is based upon the simplest imaginable
energy estimate. The condition $\l\ge 0$ in (\ref{NSSlip}) is somehow natural when this initial-boundary value problem is considered as some 
scaling limit of the Boltzmann equation of the kinetic theory of gases.

Another question is whether the condition $\l(\nu)\to 0$ as $\nu\to 0$ is optimal. Considers instead the Dirichlet boundary condition for $u_\nu$, 
i.e.
$$
u_\nu\rstr_{\d\Om}=0\,.
$$
Formally, this boundary condition corresponds with any one of the slip boundary conditions above with 
$$
\varliminf_{\nu\to 0}\l(\nu)>0\,.
$$
In that case, it well known that the Euler equations (\ref{IncEul}) may fail to describe the inviscid limit of the Navier-Stokes equations, even in the 
simpler 2 dimensional case. Because the Dirichlet boundary condition overdetermines the velocity field in the inviscid limit, the Euler equations 
(\ref{IncEul}) are expected to govern the inviscid limit of the Navier-Stokes equations only if the effect of viscosity remains confined on a thin layer 
near the boundary. But it may happen --- and does happen under certain circumstances --- that the viscous layer detaches from the boundary, as 
for instance in the case of the so-called von Karman vortex streets in the case of a Navier-Stokes flow past a cylinder, even at moderate Reynolds 
numbers. While this situation seems beyond the grasp of current mathematical analysis, there exists a least a very interesting criterion due to 
T. Kato \cite{KatoBL}, formulated in terms of the viscous energy dissipation only, identifying situations where the inviscid limit of the incompressible 
Navier-Stokes equations with  Dirichlet boundary condition is described by the Euler equations. This suggests that, unless $\l(\nu)\to 0$, the Euler 
equations (\ref{IncEul}) might also fail to govern the inviscid limit of the Navier-Stokes equations with slip boundary conditions (\ref{NSSlip}). 

\section{From the Boltzmann Equation with Accomodation Boundary Condition to the Navier-Stokes Equations with Slip Boundary Conditions}
\lb{S-BNS}

In this section, we revisit the incompressible Navier-Stokes limit for the Boltzmann equation in the case of the initial-boundary value problem.
Our main interest is to understand how the slip boundary condition arises from Maxwell's accommodation boundary condition at the kinetic level in 
the fluid dynamic limit, and especially how the slip coefficient is related to the accommodation parameter. Strictly speaking, this is not needed in the 
proof of the main result in the present paper. Therefore, the discussion in this section will be only formal, along the line of \cite{BGL1}.

Consider the Boltzmann equation with the incompressible Navier-Stokes scaling
\be\lb{BoltzINS}
\eps\d_tF_\eps+v\cdot\grad_xF_\eps=\frac1\eps\cB(F_\eps,F_\eps)\,.
\ee
Here the unknown is the distribution function $F\equiv F(t,x,v)$ that is the density at time $t$ of molecules with velocity $v\in\bR^N$ at the position
$x\in\Om$ with respect to the phase space Lebesgue measure $dxdv$. 

\subsection{Formal structure of the Boltzmann equation}

The Boltzmann collision integral acts only on the $v$ variable in $F_\eps$, keeping $t,x$ as parameters. Its expression for $\phi\in C_c(\bR^N)$ is
\be\lb{CollInt}
\cB(\phi,\phi)(v)=\iint_{\bR^N\times\bS^{N-1}}(\phi(v')\phi(v'_*)-\phi(v)\phi(v_*))b(v-v_*,\om)dv_*d\om
\ee
where $v',v'_*\in\bR^N$ are the velocities of 2 identical particles about to undergo an elastic collision, assuming that their post-collision velocities 
are $v,v_*\in\bR^N$. The set of all possible pre-collision velocities $v',v'_*$ are parametrized by the unit vector $\om$ as follows:
\be\lb{vv*-v'v'*}
\left\{
\ba
v'\equiv\, v'(v,v_*,\om)&:=v-(v-v_*)\cdot\om\om\,,
\\
v'_*\equiv v'_*(v,v_*,\om)&:=v_*\!+\!(v-v_*)\cdot\om\om\,.
\ea
\right.
\ee
The collision kernel $b(z,\om)>0$ is a locally integrable function that satisfies the symmetries
\be\lb{Sym-b}
b(v-v_*,\om)=b(v_*-v,\om)=b(v'-v'_*,\om)
\ee
a.e. in $(v,v_*,\om)$, assuming that $v'$ and $v'_*$ are given in terms of $v,v_*,\om$ by the relations (\ref{vv*-v'v'*}). Depending on the growth of 
the collision kernel $b$ as $|v-v_*|\to+\infty$, the collision integral can be extended by continuity to larger classes of functions than $C_c(\bR^N)$.
Finally, we denote
$$
\cB(F,F)(t,x,v):=\cB(F(t,x,\cdot),F(t,x,\cdot))(v)\,.
$$

The collision integral satisfies the identities
\be\lb{ConsCollInt}
\left\{
\ba
\int_{\bR^N}\cB(\phi,\phi)(v)dv=0\,,
\\
\int_{\bR^N}\cB(\phi,\phi)(v)vdv=0\,,
\\
\int_{\bR^N}\cB(\phi,\phi)(v)|v|^2dv=0\,,
\ea
\right.
\ee
for each $\phi\in C_c(\bR^N)$ or in the larger class allowed by the growth at infinity of the collision kernel $b$. As a result, whenever $F$ is a
classical solution of the scaled Boltzmann equation (\ref{BoltzINS}) with appropriate decay as $|v|\to\infty$, 
\be\lb{LocCons}
\left\{
\ba
\eps\d_t\int_{\bR^N}F_\eps dv+\Div_x\int_{\bR^N}vF_\eps dv=0\,,
\\
\eps\d_t\int_{\bR^N}vF_\eps dv+\Div_x\int_{\bR^N}v^{\otimes 2}F_\eps dv=0\,,
\\
\eps\d_t\int_{\bR^N}\tfrac12|v|^2F_\eps dv+\Div_x\int_{\bR^N}v\tfrac12|v|^2F_\eps dv=0\,,
\ea
\right.
\ee
and these relations are the local conservation laws of mass, momentum and energy respectively.

Whenever $b(v-v_*,\om)$ has polynomial growth as $|v-v_*|\to\infty$, for each positive, rapidly decaying $\phi\in C(\bR^N)$ such 
that $\ln\phi$ has polynomial growth as $|v|\to\infty$, Boltzmann's H Theorem states that
\be\lb{HCollInt<}
\int_{\bR^N}\cB(\phi,\phi)(v)\ln\phi(v)dv\le 0\,,
\ee
and
\be\lb{HCollInt=}
\ba
\int_{\bR^N}\cB(\phi,\phi)(v)\ln\phi(v)dv=0&\Leftrightarrow\cB(\phi,\phi)=0
\\
&\Leftrightarrow\phi\hbox{ is a Maxwellian distribution,}
\ea
\ee
meaning that there exists $\rho,\th>0$ and $u\in\bR^N$ such that
\be\lb{Maxw}
\phi(v)=\cM_{\rho,u,\th}(v):=\frac{\rho}{(2\pi\th)^{N/2}}e^{-\frac{|v-u|^2}{2\th}}
\ee
for all $v\in\bR^N$. As a result, whenever $F$ is a classical solution of the scaled Boltzmann equation  (\ref{BoltzINS}) with appropriate decay 
as $|v|\to\infty$, it satisfies the differential entropy inequality
\be\lb{LocalHThm}
\ba
\eps\d_t\int_{\bR^N}F_\eps\ln F_\eps dv+\Div_x\int_{\bR^N}vF_\eps\ln F_\eps dv&
\\
=-\frac1\eps\int_{\bR^N}\cB(F_\eps,F_\eps)\ln F_\eps dv&\le 0\,.
\ea
\ee

Throughout this paper, we denote 
\be\lb{NotM}
M:=\cM_{1,0,1}\,.
\ee

Since Maxwellians are equilibrium distributions for the collision integral, it is natural to investigate the linearization thereof about a Maxwellian,
say $M$ for simplicity --- the case of an arbitrary Maxwellian being similar. We therefore introduce the linearized collision  operator in the form
$$
\cL_M\phi:=-2M^{-1}\frac{\de\cB(F,F)}{\de F}\rstr_{F=M}\cdot M\phi\,,
$$
i.e.
$$
\cL_M\phi(v):=\iint_{\bR^N\times\bS^{N-1}}(\phi(v)+\phi(v_*)-\phi(v')-\phi(v'_*))b(v-v_*,\om)Mdv_*d\om\,.
$$
Under certain assumptions on the collision kernel $b$, known as Grad's angular cutoff assumption, H. Grad proved in \cite{Grad62} that $\cL_M$
 is an unbounded, self-adjoint Fredholm operator on $L^2(\bR^N;Mdv)$ with domain
$$
\mathrm{D}(\cL_M):=\{\phi\in L^2(\bR^N;Mdv)\,|\,\phi(\overline{b}\star_vM)\in L^2(\bR^N;Mdv)\}\,,
$$
where
$$
\overline{b}(z):=\int_{\bS^{N-1}}b(z,\om)d\om
$$
and $\star_v$ designates the convolution product in the $v$ variable. Moreover, the nullspace of $\cL_M$ is
$$
\Ker\cL_M=\Span\{1,v_1,\ldots,v_n,|v|^2\}\,.
$$
In particular, the tensor field $A(v)=v^{\otimes 2}-\tfrac1N|v|^2$ satisfies $A\bot\Ker\cL_M$, so that, by the Fredholm alternative, there exists a
unique tensor field 
$$
\hat A\in\mathrm{D}(\cL_M)\cap(\Ker\cL)^\bot\quad\hbox{ such that }\cL_M\hat A=A
$$
componentwise.

Henceforth in this paper, we assume that the collision kernel $b$ comes from a hard cutoff potential in the sense of Grad, and more precisely 
that it satisfies, for some constant $C_b>0$ and all $(z,\om)\in\bR^N\times\bS^{N-1}$,
\be\lb{HardCut}
0<b(z,\om)\le C_b(1+|z|)\,,\quad\hbox{ and }\overline{b}(z)\ge\frac1{C_b}\,.
\ee

\subsection{Boundary value problem and fluid dynamic limit}

The incompressible Navier-Stokes limit of the Boltzmann equation bears on solutions of the Boltzmann equation that are of the form
\be\lb{NSScalF}
F_\eps=M(1+\eps g_\eps)\,,
\ee
where it is understood that the relative number density fluctuation $g_\eps$ is $O(1)$ in some sense to be made precise as $\eps\to 0$: see
\cite{BGL0,BGL1,BGL2} for more details, together with physical justifications for this scaling assumption.

Here, the scaled Boltzmann equation (\ref{BoltzINS}) is set on the spatial domain $\Om$, with Maxwell's accommodation at the boundary, that
is assumed to be maintained at the constant temperature $1$. This boundary condition reads
\be\lb{Accom}
\ba
F_\eps(t,x,v)=(1-\a)\cR_x F_\eps(t,x,v)+\a\L_x\left(\frac{F_\eps}{M}\right)(t,x)M(v)\,,
\\
x\in\d\Om\,,\,\,v\cdot n_x<0\,,
\ea
\ee
where 
\be\lb{DefR}
\cR_x F(t,x,v):=F(t,x,v-2v\cdot n_xn_x)
\ee
is the specular reflection operator on the boundary, while
\be\lb{DefLambd}
\L_x\phi\,:=\sqrt{2\pi}\int_{\bR^N}\phi(v)(v\cdot n_x)_+M(v)dv\,.
\ee
In (\ref{Accom}), the parameter $\a$ satisfies $0\le\a\le 1$, and is called the accommodation coefficient. The case $\a=0$ corresponds with 
specular reflection of the gas molecules on $\d\Om$ without thermal exchange, while the case $\a=1$ corresponds with diffuse reflection, 
or total accommodation, in which case gas molecules are instantaneously thermalized at the boundary.

Henceforth, we denote, for each $\phi\in L^1(\bR^N;Mdv)$, 
$$
\la\phi\ra:=\int_{\bR^N}\psi(v)M(v)dv\,.
$$

\begin{Thm}\lb{T-FormNSLim}
Let $(F_\eps)_{\eps>0}$ be a family of solutions of the scaled Boltzmann equation (\ref{BoltzINS}) set on the spatial domain $\Om$, satisfying
the accommodation boundary condition (\ref{Accom}) on $\d\Om$. Assume that the relative fluctuations
$$
g_\eps=\frac{F_\eps-M}{\eps M}\to g
$$
a.e. and in weak-$L^1_{loc}(\bR_+\times\Om\times\bR^N;dtdxMdv)$ (possibly up to extraction of a subsequence), and that
$$
\la|v|^3\indc_{|v|>R}|g_\eps|\ra+\la|\hat A|\indc_{|v|>R}|\cQ(g_\eps,g_\eps)|\ra+\la|\hat A||v|\indc_{|v|>R}|g_\eps|\ra\to 0
$$
in $L^1_{loc}(\bR_+\times\Om)$ as $R\to+\infty$ uniformly in $\eps>0$. Then 
$$
g(t,x,v)=\rho(t,x)+u(t,x)\cdot v+\th(t,x)\tfrac12(|v|^2-N)\,,
$$
where $u$ is a solution of the incompressible Navier-Stokes equations
$$
\left\{
\ba
{}&\Div_xu=0\,,
\\
&\d_tu+\Div_x(u^{\otimes 2})+\grad_xp=\nu\Dlt_xu\,,
\ea
\right.
$$
and where 
$$
\nu=\tfrac{1}{(N-1)(N+2)}\la\hat A:A\ra\,.
$$
Assume further that
$$
\left\{
\ba
{}\la v\cdot n_xg_\eps\rstr_{\d\Om}\ra&\to\la v\cdot n_xg\rstr_{\d\Om}\ra
\\
\la v_\tau (v\cdot n_x)_+g_\eps\rstr_{\d\Om}\ra&\to\la v_\tau (v\cdot n_x)_+g\rstr_{\d\Om}\ra
\ea
\right.
$$
in weak-$L^1_{loc}(\bR_+\times\d\Om)$. Then, the velocity field $u$ satisfies the boundary condition
$$
\left\{
\ba
{}&u\cdot n_x=0\,,&&\quad x\in\d\Om\,,
\\
&\nu(\Si(u)\cdot n_x)_\tau+\l u=0\,,&&\quad x\in\d\Om\,,
\ea
\right.
$$
where the slip coefficient is given by the formula
$$
\l=\tfrac{\a_0}{N-1}\la|v_\tau |^2(v\cdot n_x)_+\ra=\frac{\a_0}{\sqrt{2\pi}}\,.
$$
\end{Thm}

\begin{proof}
Observe that, under the substitution $v\mapsto w=v-2v\cdot nn$, one has, for each unit vector $n$,
$$
\int_{\bR^N}\phi(v-2v\cdot nn)(v\cdot n)_-Mdv=\int_{\bR^N}\phi(w)(w\cdot n)_+Mdw\,,
$$
so that, for each $x\in\d\Om$, one has
$$
\ba
\int_{\bR^N}F_\eps(t,x,v)v\cdot n_xdv&=-\a\int_{\bR^N}F_\eps(t,x,v)(v\cdot n_x)_+dv
\\
&+\a\L_x\left(\frac{F_\eps}{M}\right)(t,x)\int_{\bR^N}M(v)(v\cdot n_x)_-dv=0\,,
\ea
$$
since
$$
\int_{\bR^N}M(v)(v\cdot n_x)_+dv=\int_{\bR^N}M(v)(v\cdot n_x)_-dv=\frac1{\sqrt{2\pi}}\,.
$$
Hence, for each $x\in\d\Om$ and each $\eps>0$, one has
$$
\la vg_\eps\ra(t,x)\cdot n_x=\frac1{\eps}\int_{\bR^N}F_\eps(t,x,v)v\cdot n_xdv=0\,,
$$
so that, after passing to the limit as $\eps\to 0$, 
$$
u(t,x)\cdot n_x=\la vg\ra(t,x)\cdot n_x=0\,,\quad t>0\,,\,\,x\in\d\Om\,.
$$

Next write the local conservation of momentum --- the second local conservation law in (\ref{LocCons}) in the form
$$
\d_t\la vg_\eps\ra+\Div_x\frac1\eps\la Ag_\eps\ra+\grad_x\frac1\eps\la\tfrac1N|v|^2g_\eps\ra=0\,,
$$
where
$$
A\equiv A(v):=v^{\otimes 2}-\tfrac1N|v|^2\,.
$$
Let now $w\equiv w(x)\in\bR^N$ designate a compactly supported $C^1$ vector field on $\bR^N$ satisfying
$$
\Div w=0\,,\quad\hbox{ and }w(x)\cdot n_x=0\,,\,\,x\in\d\Om\,.
$$
Taking the inner product of both sides of the local conservation of momentum with $w$ and integrating over $\Om$ leads to
\be\lb{ConsMomIntegr}
\ba
\d_t\int_\Om w\cdot\la vg_\eps\ra dx&+\int_{\d\Om} w\otimes n_x:\frac1\eps\la Ag_\eps\ra dS(x)
\\
&-\int_{\Om} \grad w:\frac1\eps\la Ag_\eps\ra dx=0\,,
\ea
\ee
since
$$
\ba
\int_\Om w\cdot\grad_x\frac1\eps\la\tfrac1N|v|^2g_\eps\ra dx&=\int_\Om\Div_x\left(\frac1\eps\la\tfrac1N|v|^2g_\eps\ra\cdot w\right)dx
\\
&=\int_{\d\Om}\frac1\eps\la\tfrac1N|v|^2g_\eps\ra w\cdot n_xdS(x)=0\,.
\ea
$$

Next we pass to the limit in each term appearing in (\ref{ConsMomIntegr}): following the analysis in \cite{BGL0,BGL1}, one finds that
$$
\int_\Om w\cdot\la vg_\eps\ra dx\to\int_\Om w\cdot\la vg\ra dx=\int_\Om w\cdot udx
$$
while
$$
\ba
\int_\Om \grad w:\frac1\eps\la Ag_\eps\ra dx&\to\int_\Om\grad w:(A(u)-\nu\Si(u))dx
\\
&=\int_\Om\grad w:(u^{\otimes 2}-\nu\Si(u))dx\,.
\ea
$$
(Indeed, since $w$ is divergence-free, $\grad w:\tfrac1N|u|^2=\tfrac1N|u|^2\Div_xu=0$.)

It remains to analyze the boundary term
$$
\int_{\d\Om} w\otimes n_x:\frac1\eps\la Ag_\eps\ra dS(x)\,.
$$
Since $w\cdot n_x=0$ on $\d\Om$, 
$$
w\otimes n_x:\la Ag_\eps\ra=w\otimes n_x:\la v^{\otimes 2}g_\eps\ra=\la v_\tau v\cdot n_xg_\eps\ra\cdot w\,.
$$
At this point, we decompose the boundary term into the contribution of gas molecules about to collide and those having just collided with
the boundary
$$
\la v_\tau v\cdot n_x g_\eps\ra=\la v_\tau v\cdot n_x \indc_{v\cdot n_x>0}g_\eps\ra+\la v_\tau v\cdot n_x \indc_{v\cdot n_x<0}g_\eps\ra
$$
and use the accommodation condition to write
$$
\la v_\tau v\cdot n_x \indc_{v\cdot n_x<0}g_\eps\ra=\la v_\tau v\cdot n_x \indc_{v\cdot n_x<0}((1-\a)\cR_xg_\eps+\a\L_x(g_\eps))\ra\,.
$$
Observing that
$$
\ba
\la v_\tau v\cdot n_x \indc_{v\cdot n_x<0}\cR_xg_\eps\ra&=\la\cR_x(v_\tau v\cdot n_x \indc_{v\cdot n_x>0})g_\eps\ra
\\
&=-\la v_\tau v\cdot n_x \indc_{v\cdot n_x>0}((1-\a)g_\eps+\a\L_x(g_\eps))\ra\,,
\ea
$$
we conclude that
$$
\la v_\tau v\cdot n_x \indc_{v\cdot n_x<0}g_\eps\ra=-\la v_\tau v\cdot n_x \indc_{v\cdot n_x>0}((1-\a)g_\eps+\a\L(g_\eps))\ra
$$
so that 
$$
\la v_\tau v\cdot n_x g_\eps\ra=\a\la v_\tau v\cdot n_x\indc_{v\cdot n_x>0}(g_\eps-\L_x(g_\eps))\ra=\a\la v_\tau v\cdot n_x\indc_{v\cdot n_x>0}g_\eps\ra
$$
--- where the second equality follows from the fact that the function $v\mapsto v_\tau v\cdot n_x\indc_{v\cdot n_x>0}(g_\eps-\L_x(g_\eps)$ is odd
in $v_\tau $. 

Therefore the boundary term appearing in (\ref{ConsMomIntegr}) becomes
$$
\int_{\d\Om} w\otimes n_x:\frac1\eps\la Ag_\eps\ra dS(x)=\frac{\a}{\eps}\int_{\d\Om} w\cdot\la v_\tau v\cdot n_x\indc_{v\cdot n_x>0}g_\eps\ra dS(x)\,.
$$
Assume that $\a\equiv\a(\eps)$ varies with $\eps$ so that $\a(\eps)/\eps\to\a_0$ as $\eps\to 0$. Since 
$$
g_\eps\to g=\rho+u\cdot v+\th\tfrac12(|v|^2-N)
$$
and we already know that 
$$
u\cdot n_x=0\quad\hbox{ on }\d\Om\,,
$$
one has
$$
\ba
w\cdot\la v_\tau v\cdot n_x\indc_{v\cdot n_x>0}g_\eps\ra&\to\la v_\tau ^{\otimes 2}(v\cdot n_x)_+\ra:u_t\otimes w
\\
&=\tfrac1{N-1}\la|v_\tau |^2(v\cdot n_x)_+\ra u\cdot w\,.
\ea
$$

Thus, passing to the limit in (\ref{ConsMomIntegr}) leads to
\be\lb{MotEqIntegr}
\ba
\d_t\int_\Om w\cdot u dx&+\tfrac{\a_0}{N-1}\la|v_\tau |^2(v\cdot n_x)_+\ra\int_{\d\Om} u\cdot w dS(x)
\\
&-\int_{\Om} \grad w:(u^{\otimes 2}-\nu\Si(u))dx=0\,.
\ea
\ee
(Notice that the term $\la|v_\tau |^2(v\cdot n_x)_+\ra$ is independent of $x$ and therefore comes out of the boundary integral.)

Whenever $u(t,\cdot)\in C^2(\overline\Om)$, applying Green's formula transforms the last integral above into
$$
\ba
\int_\Om\grad w:(u^{\otimes 2}-\nu\Si(u))dx=&-\int_\Om w\cdot(\Div_x(u^{\otimes 2})-\nu\Div_x(\Si(u)))dx
\\
&+\int_{\d\Om}w\cdot uu\cdot n_xdS(x)
\\
&-\nu\int_{\d\Om}w\cdot(\Si(u)\cdot n)dS(x)\,.
\ea
$$
Since $u\cdot n_x=0$ on $\d\Om$, the second integral on the right-hand side above vanishes, and since $\Div_xu=0$, one has
$\Div_x(\Si(u))=\Dlt_xu$, so that
$$
\ba
\int_\Om\grad w:(u^{\otimes 2}-\nu\Si(u))dx=&-\int_\Om w\cdot(\Div_x(u^{\otimes 2})-\nu\Dlt_xu)dx
\\
&-\nu\int_{\d\Om}w\cdot(\Si(u)\cdot n)dS(x)\,.
\ea
$$
Thus, if $u\in C^2([0,T]\times\overline\Om)$, the equality (\ref{MotEqIntegr}) becomes
$$
\ba
\d_t&\int_\Om w\cdot u dx+\int_\Om w\cdot(\Div_x(u^{\otimes 2})-\nu\Dlt_xu)dx
\\
&+\tfrac{\a_0}{N-1}\la|v_\tau |^2(v\cdot n_x)_+\ra\int_{\d\Om} u\!\cdot\!w dS(x)\!+\!\nu\int_{\d\Om}w\!\cdot\!(\Si(u)\!\cdot\! n)dS(x)=0\,.
\ea
$$
This identity holds, say, for each $w\in C^\infty_c(\overline\Om;\bR^N)$. In particular, it holds for each $w\in C^\infty_c(\Om;\bR^N)$, 
which implies that
$$
\d_tu+\Div_x(u^{\otimes 2})-\nu\Dlt_xu=-\grad_xp
$$
in the sense of distributions (for some $p\in\cD'(\bR_+^*\times\Om)$). Since the velocity field $u\in C^2([0,T]\times\overline\Om)$ we 
conclude that $p\in C^1([0,T]\times\overline\Om)$. Substituting this in the identity above with $w\in C^\infty_c(\overline\Om;\bR^N)$ 
gives
$$
-\int_\Om w\cdot\grad_xpdx+\int_{\d\Om}(\nu\Si(u)\cdot n_x+\tfrac{\a_0}{N-1}\la|v_\tau |^2(v\cdot n_x)_+\ra u)\cdot wdS(x)=0
$$
and since, by Green's formula,
$$
-\int_\Om w\cdot\grad_xpdx=-\int_\Om\Div_x(pw)dx=\int_{\d\Om}pw\cdot n_xdS(x)=0\,,
$$
we conclude that 
$$
\nu(\Si(u)\cdot n_x)_\tau+\tfrac{\a_0}{N-1}\la|v_\tau |^2(v\cdot n_x)_+\ra u=0\quad\hbox{ on }\d\Om\,.
$$

In other words, (\ref{MotEqIntegr}) is the weak formulation of
$$
\left\{
\ba
{}&\d_tu+\Div_x(u^{\otimes 2})-\nu\Dlt_xu=-\grad_xp\,,&&\quad x\in\Om\,,&&\,\,t>0\,,
\\
&\nu(\Si(u)\cdot n_x)_\tau+\l u=0\,,&&\quad x\in\d\Om\,,&&\,\,t>0\,,
\\
&u\cdot n_x=0\,,&&\quad x\in\d\Om\,,&&\,\,t>0\,. 
\ea
\right.
$$
with 
$$
\l=\tfrac{\a_0}{N-1}\la|v_\tau |^2(v\cdot n_x)_+\ra=\frac{\a_0}{\sqrt{2\pi}}\,.
$$
\end{proof}

\smallskip
The argument above is a proof of the (formal) Navier-Stokes limit Theorem \ref{T-FormNSLim} by a moment method analogous to the one used
in \cite{BGL1}. As far as we know, the first derivation of this slip boundary condition, in the steady, linearized regime (i.e. leading to the Stokes 
equations in the fluid limit), is due to K. Aoki, T. Inamuro and Y. Onishi \cite{Aoki} (see especially formula (33) in that reference). That derivation 
uses a Hilbert expansion method (formal series expansion of the solution of the Boltzmann equation in powers of the Knudsen number $\eps$). 
The interested reader is referred to the recent book by Y. Sone \cite{SoneBk2} (in particular to \S 3.7 there) for a systematic study of boundary 
conditions in the context of the fluid dynamic limit of the Boltzmann equation.

For a complete proof of the derivation of the same slip boundary condition as in Theorem \ref{T-FormNSLim} in the linearized regime --- i.e. in a
situation where the limiting equation is the Stokes, instead of the Navier-Stokes equations --- the reader is referred to the work of N. Masmoudi 
and L. Saint-Raymond \cite{MasSR}.

\section{From the Boltzmann Equation with Accomodation Boundary Condition to the Incompressible Euler Equations}\lb{S-BEul}

In this section, we consider the Boltzmann equation in the scaling leading to the incompressible Euler equations in the fluid dynamic limit.
We recall from \cite{BGL0,BGL1} that this scaling is
\be\lb{BoltzIEul}
\eps\d_tF_\eps+v\cdot\grad_xF_\eps=\frac1{\eps^{1+q}}\cB(F_\eps,F_\eps)\,,\quad (x,v)\in\Om\times\bR^N\,,
\ee
with $q>0$, while the distribution function $F_\eps$ is sought in the same form (\ref{NSScalF}) as in the Navier-Stokes limit. This scaled
Boltzmann equation is supplemented with Maxwell's accommodation condition on $\d\Om$, with accommodation coefficient $\a\equiv\a(\eps)$ 
driven by the small parameter $\eps$:
\be\lb{AccomEul}
\ba
F_\eps(t,x,v)=(1-\a(\eps))\cR_x F_\eps(t,x,v)+\a(\eps)\L\left(\frac{F_\eps}{M}\right)(t,x)M(v)\,,
\\
x\in\d\Om\,,\,\,v\cdot n_x<0\,,
\ea
\ee
and with the initial condition
\be\lb{BoltzCondin}
F_\eps(0,x,v)=F^{in}_\eps(x,v)\,,\quad (x,v)\in\Om\times\bR^N\,.
\ee

The formal result presented in Theorem \ref{T-FormNSLim} suggests that, in the limit as $\eps\to 0$, the velocity field
$$
\lim_{\eps\to 0}\frac1\eps\int_{\bR^N}vF_\eps dv
$$
should behave like the solution of the incompressible Navier-Stokes equations with kinematic viscosity of order $\eps^q$ and with slip
boundary condition with slip coefficient of the order of $\a(\eps)/\eps$. Thus, if $\a(\eps)=o(\eps)$, Theorem \ref{T-Eul<NS} suggests that
this velocity field should satisfy the incompressible Euler equations (\ref{IncEul}). In fact, the formal result in Theorem \ref{T-FormNSLim} 
is only a guide for our intuition, and we shall give a direct proof of the Euler limit starting from the Boltzmann equation with accommodation
boundary condition without using the Navier-Stokes limit.

\subsection{Renormalized solutions and a priori estimates}

Global solutions of the Cauchy problem for the Boltzmann equation for initial data of arbitrary size have been constructed by R. DiPerna and 
P.-L. Lions \cite{dPL}. Their theory of renormalized solutions was extended to the initial boundary value problem by S. Mischler \cite{Misch}.
His result is summarized below --- see also section 2 in \cite{MasSR} and section 2.3.2 of \cite{LNLSR}.

\begin{Thm}[Mischler]\lb{T-ThmMisch}
Let $F^{in}_\eps\equiv F^{in}_\eps(x,v)\ge 0$ a.e. on $\Om\times\bR^N$ be a measurable function satisfying
$$
\iint_{\Om\times\bR^N}(1+|v|^2+|\ln F^{in}_\eps(x,v)|)F^{in}_\eps(x,v)dxdv<+\infty\,.
$$
There exists $F_\eps\in C(\bR_+;L^1(\Om\times\bR^N))$ satisfying the initial condition (\ref{BoltzCondin}), and the Boltzmann equation 
(\ref{BoltzIEul}) together with the boundary condition (\ref{AccomEul}) in the renormalized sense, meaning that, for each $\G\in C^1(\bR_+)$ 
such that $Z\mapsto\sqrt{1+Z}\G'(Z)$ is bounded on $\bR_+$, the function 
$$
\G'\left(\frac{F_\eps}{M}\right)\cB(F_\eps,F_\eps)\in L^1_{loc}(\bR_+\times\overline\Om\times\bR^N)
$$
and
$$
\ba
\int_0^\infty\iint_{\Om\times\bR^N}\G\left(\frac{F_\eps}{M}\right)(\eps\d_t+v\cdot\grad_x)\phi Mdvdxdt
\\
+\frac1{\eps^{1+q}}\int_0^\infty\iint_{\Om\times\bR^N}\G'\left(\frac{F_\eps}{M}\right)\cB(F_\eps,F_\eps)\phi dvdxdt
\\
=
\int_0^\infty\iint_{\d\Om\times\bR^N}\G\left(\frac{F_\eps}{M}\right)\phi v\cdot n_xMdvdS(x)dt
\\
-\eps\iint_{\Om\times\bR^N}\G\left(\frac{F^{in}_\eps}{M}\right)\phi\rstr_{t=0}Mdvdx
\ea
$$
for each $\phi\in C^1_c(\bR_+\times\overline\Om\times\bR^N)$. 

Moreover

\smallskip
\noindent
a) the trace of $F_\eps$ on $\d\Om$ satisfies the accommodation boundary condition
$$
F_\eps\rstr_{\d\Om}(t,x,v)=(1-\a)\cR_x(F_\eps\rstr_{\d\Om})(t,x,v)+\L_x\left(\frac{F_\eps\rstr_{\d\Om}}{M}\right)(t,x)M(v)
$$
for a.e. $(t,x,v)\in\bR_+\times\d\Om\times\bR^N$ such that $v\cdot n_x>0$;

\noindent
b) the distribution function $F_\eps$ satisfies the local conservation law of mass
$$
\eps\d_t\int_{\bR^N}F_\eps dv+\Div_x\int_{\bR^N}vF_\eps dv=0
$$
with boundary condition
$$
\int_{\bR^N}F_\eps(t,x,v)v\cdot n_xdv=0\,,\quad x\in\d\Om\,,\,\,t>0\,;
$$

\noindent
c) the  distribution function $F_\eps$ satisfies the relative entropy inequality
$$
\ba
H(F_\eps|M)(t)&-H(F^{in}_\eps|M)
\\
\le&-\frac1{\eps^{2+q}}\int_0^t\int_\Om\cP_\eps(s,x)dxds-\frac\a\eps\int_0^t\int_{\d\Om}\cD\cG_\eps(s,x)dxds
\ea
$$
for each $t>0$, where the following notations have been used: for each $f,g$ measurable on $\Om\times\bR^N$ such that $f\ge 0$ and $g>0$ a.e. ,
the relative entropy is
$$
H(f|g):=\iint_{\Om\times\bR^N}h\left(\frac{f}{g}-1\right)gdvdx
$$
with
$$
h(z):=(1+z)\ln(1+z)-z\,,
$$
while the entropy production rate per unit volume is
$$
\cP_\eps:=\iiint_{\bR^N\times\bR^N\times\bS^{N-1}}r\left(\frac{F'_\eps F'_{\eps *}}{F_\eps F_{\eps *}}-1\right)F_\eps F_{\eps*}b(v-v_*,\om)dvdv_*d\om
$$
with
$$
r(z):=z\ln(1+z)\ge 0\,,
$$
and the Darrozes-Guiraud information is
$$
\cD\cG_\eps:=\frac1{\sqrt{2\pi}}\left(\L_x\left(h\left(\frac{F_\eps}{M}-1\right)\right)-h\left(\L_x\left(\frac{F_\eps}{M}-1\right)\right)\right)\,.
$$
In particular $H(f|g)\ge 0$ since $h\ge 0$ on $[-1,+\infty)$ and $\cP_\eps\ge 0$ a.e. on $\bR_+\times\Om$ since $r\ge 0$ on $(-1,+\infty)$, while
$\cD\cG_\eps\ge 0$ a.e. on $\bR_+\times\d\Om$ since $h$ is convex and $\L_x$ is the average with respect to a probability measure;

\noindent
d) for each $T>0$ and each compact $K\subset\d\Om$, there exists $C_{K,T}>0$ such that, for each $\eps>0$, one has
$$
\int_0^T\iint_{K\times\bR^N}F_\eps(t,x,v)(v\cdot n_x)^2M(v)dvdS(x)dt\le C_{K,T}\,,\quad \eps>0\,.
$$
\end{Thm}

Statement d) appears in \cite{MasSR}, without proof. We give a brief justification for this estimate below.

\smallskip
Notice that in general, renormalized solutions of the initial-boundary value problem (\ref{BoltzIEul})-(\ref{AccomEul})-(\ref{BoltzCondin}) are not
known to satisfy the local conservation law of momentum --- see equation (2.35) in \cite{LNLSR} for a variant involving a defect measure, following
an earlier remark due to P.-L. Lions and N. Masmoudi.

At variance, any classical solution $F_\eps\in C(\bR_+\times\overline\Om\times\bR^N)$ is continuously differentiable in $(t,x)$ and such that 
\be\lb{DecayAssump}
\ba
v\mapsto\sup_{0\le t\le T\atop|x|\le R}(|F_\eps(t,x,v)|+|\d_tF_\eps(t,x,v)|+|\grad_xF_\eps(t,x,v)|)&
\\
\hbox{ is rapidly decaying as $|v|\to+\infty$}&\,,
\ea
\ee
one has
$$
\eps\d_t\int_{\bR^N}vF_\eps dv+\Div_x\int_{\bR^N}v^{\otimes 2}F_\eps dv=0\,,\quad x\in\Om\,,\,\,t>0\,.
$$
Moreover, for each $w\in C^1_c(\bR_+\times\overline\Om)$, Green's formula implies that
\be\lb{ConsMom+Green}
\ba
\int_0^t\iint_{\Om\times\bR^N}(\eps v\cdot\d_tw(s,x)+v^{\otimes 2}:\grad_xw(s,x))F_\eps(s,x,v)dvdxds&
\\
=\int_0^t\int_{\d\Om\times\bR^N}v^{\otimes 2}:w(s,x)\otimes n_xF_\eps(s,x,v)dvdS(x)ds&
\\
+\eps\iint_{\Om\times\bR^N}w(t,x)\cdot vF_\eps(t,x,v)dvdx&
\\
-\eps\iint_{\Om\times\bR^N}w(0,x)\cdot vF^{in}_\eps(x,v)dvdx&\,.
\ea
\ee
Let us use the accommodation condition (\ref{AccomEul}) to reduce the boundary integral:
$$
\ba
\int_{\bR^N}v^{\otimes 2}:w(t,x)\otimes n_xF_\eps(s,x,v)dv
\\
=
\int_{\bR^N}(w(s,x)\!\cdot\!v)(v\!\cdot\!n_x)_+F_\eps(s,x,v)dv
\\
-
\int_{\bR^N}(w(s,x)\!\cdot\!v)(v\!\cdot\!n_x)_-F_\eps(s,x,v)dv
\ea
$$
and, whenever $w$ is tangential on $\d\Om$, one has
$$
\ba
\int_{\bR^N}(w(s,x)\!\cdot\!\!\cdot\!n_x)_-F_\eps(s,x,v)dv&
\\
=
\int_{\bR^N}(w(s,x)\!\cdot\!v)(v\!\cdot\!n_x)_-\left((1-\a)\cR_xF_\eps+\a\L_x\left(\frac{F_\eps}{M}\right)M\right)(s,x,v)dv&
\\
=
\int_{\bR^N}(w(s,x)\!\cdot\!v)(v\!\cdot\!n_x)_+\left((1-\a)F_\eps+\a\L_x\left(\frac{F_\eps}{M}\right)M\right)(s,x,v)dv&\,.
\ea
$$
Therefore
$$
\ba
\int_{\bR^N}v^{\otimes 2}:w(s,x)\otimes n_xF_\eps(s,x,v)dv
\\
=
\a\int_{\bR^N}(w(s,x)\cdot v)(v\cdot n_x)_+F_\eps(s,x,v)dv
\\
-
\a\L_x\left(\frac{F_\eps}{M}\right)\int_{\bR^N}(w(s,x)\cdot v)(v\cdot n_x)_+Mdv
\ea
$$
and the last integral vanishes since the integrand is odd in the tangential component of $v$. 

Finally, whenever $w\in C^1_c(\bR_+\times\overline\Om)$ is tangential on $\d\Om$, one has
$$
\ba
\int_{\bR^N}v^{\otimes 2}:w(s,x)\otimes n_xF_\eps(s,x,v)dv
\\
=\a\int_{\bR^N}(w(s,x)\cdot v)(v\cdot n_x)_+F_\eps(s,x,v)dv\,.
\ea
$$
Therefore, each classical solution $F_\eps\in C(\bR_+\times\overline\Om\times\bR^N)$ of the initial-boundary value problem 
(\ref{BoltzIEul})-(\ref{AccomEul})-(\ref{BoltzCondin}) that is continuously differentiable in $(t,x)$ and satisfies (\ref{DecayAssump}) also verifies
 \be\lb{ConsMom}
\ba
\int_0^t\iint_{\Om\times\bR^N}(\eps v\cdot\d_tw(s,x)+v^{\otimes 2}:\grad_xw(s,x))F_\eps(s,x,v)dvdxds& 
\\
=\a\int_0^t\int_{\d\Om\times\bR^N}(w(s,x)\cdot v)(v\cdot n_x)_+F_\eps(s,x,v)dvdS(x)ds&
\\
+\eps\iint_{\Om\times\bR^N}w(t,x)\cdot vF_\eps(t,x,v)dvdx&
\\
-\eps\iint_{\Om\times\bR^N}w(0,x)\cdot vF^{in}_\eps(x,v)dvdx&\,.
\ea
\ee

Henceforth, we shall consider exclusivey renormalized solutions of the initial-boundary value problem 
(\ref{BoltzIEul})-(\ref{AccomEul})-(\ref{BoltzCondin}) satisfying the identity (\ref{ConsMom}) for each vector field $w\in C^1_c(\bR_+\times\overline\Om)$ 
tangential on $\d\Om$ and such that $\Div_xw=0$ on $\Om$.

Now for estimate d) in Theorem \ref{T-ThmMisch}.

\begin{proof}[Proof of estimate d)]
A renormalized solution of the initial-boundary value problem (\ref{BoltzIEul})-(\ref{AccomEul})-(\ref{BoltzCondin}) can be constructed as the
limit for $\eps>0$ fixed and $m\to+\infty$, of solutions $F_{\eps,m}$ of the approximating equation
\be\lb{BoltzIEulApprox}
\eps\d_tF_{\eps,m}+v\cdot\grad_xF_{\eps,m}=\frac1{\eps^{1+q}}\frac{\cB_m(F_{\eps,m},F_{\eps,m})}{1+\displaystyle\frac1m\int_{\bR^N}F_{\eps,m}dv}\,,
	\quad (x,v)\in\Om\times\bR^N\,,
\ee
with the same initial and boundary conditions (\ref{AccomEul})-(\ref{BoltzCondin}) satisfied by $F_{\eps,m}$, where the approximate collision 
integral is given by the same expression as Boltzmann's collision integral with collision kernel $b$ replaced with its truncated variant $b_m$
defined as\footnote{For all $x,y\in\bR$, the notation $x\wedge y$ designates $\min(x,y)$.}
$$
b_m(v-v_*,\om):=m\wedge b(v-v_*,\om)\,.
$$

Let $\xi$ be a compactly supported $C^1$ vector field satisfying
$$
\xi(x)=a(x)n_x\hbox{ for each }x\in\d\Om\,,\hbox{ with }a\ge 0\hbox{ on }\d\Om\hbox{ and }a=1\hbox{ on }K\,.
$$
Since the approximate collision integral in (\ref{BoltzIEulApprox}) is normalized with an \textit{average} of $F_{\eps, m}$ with respect to $v$, all
solutions of that equation satisfy equality (\ref{ConsMom+Green}) for any $w\in C^1_c(\bR_+\times\bR^N)$, i.e. the local conservation of momentum.
In other words,
$$
\ba
\int_0^t\iint_{\Om\times\bR^N}(v^{\otimes 2}:\grad_x\xi(x))F_{\eps,m}(s,x,v)dvdxds&
\\
=\int_0^t\iint_{\d\Om\times\bR^N}(v\cdot\xi(x))(v\cdot n_x)F_{\eps,m}(s,x,v)dvdS(x)ds&
\\
+\eps\iint_{\Om\times\bR^N}\xi(x)\cdot vF_{\eps,m}(t,x,v)dvdx&
\\
-\eps\iint_{\Om\times\bR^N}\xi(x)\cdot vF^{in}_{\eps,m}(x,v)dvdx&\,.
\ea
$$
Therefore, since $F_{\eps,m}\ge 0$ a.e. and $(v\cdot\xi(x))(v\cdot n_x)=a(x)(v\cdot n_x)^2\ge 0$ for each $x\in\d\Om$ and $v\in\bR^N$,
one has
\be\lb{ConsMom-xi}
\ba
0\le\int_0^t\iint_{K\times\bR^N}(v\cdot n_x)^2F_{\eps,m}(s,x,v)dvdS(x)ds&
\\
\le
\int_0^t\iint_{\d\Om\times\bR^N}(v\cdot\xi(x))(v\cdot n_x)F_{\eps,m}(s,x,v)dvdS(x)ds&
\\
=
\int_0^t\iint_{\Om\times\bR^N}(v^{\otimes 2}:\grad_x\xi(x))F_{\eps,m}(s,x,v)dvdxds&
\\
+\eps\iint_{\Om\times\bR^N}\xi(x)\cdot vF^{in}_{\eps,m}(x,v)dvdx&
\\
-\eps\iint_{\Om\times\bR^N}\xi(x)\cdot vF_{\eps,m}(t,x,v)dvdx\,.&
\ea
\ee

At this point, we recall that the function $h:\,z\mapsto(1+z)\ln(1+z)-z$ introduced in Theorem \ref{T-ThmMisch} has Legendre dual
$$
h^*(\zeta)=\sup_{z>-1}(\zeta z-h(z))=e^{\zeta}-\zeta-1\,.
$$
By Young's inequality --- or equivalently, by definition of $h^*$ (see for instance \cite{BGL2}) --- one has
$$
\tfrac14(1+|v|^2)\eps|g_{\eps,n}|\le h(\eps|g_{\eps,n}|)+h^*(\tfrac14(1+|v|^2))
$$
so that, for each nonnegative $\chi\in C_c(\bR^N)$, 
\be\lb{Young<}
\ba
\iint_{\Om\times\bR^N}&\chi(x)(1+|v|^2)F_{\eps,m}(t,x,v)dvdx
\\
&=\iint_{\Om\times\bR^N}\chi(x)(1+|v|^2)(1+\eps g_{\eps,m})(t,x,v)M(v)dvdx
\\
&\le\iint_{\Om\times\bR^N}\chi(x)(1+h^*(\tfrac14(1+|v|^2))M(v)dvdx
\\
&+\|\chi\|_{L^\infty}H(F_{\eps,m}|M)(t)
\\
&\le\iint_{\Om\times\bR^N}\chi(x)(1+h^*(\tfrac14(1+|v|^2))M(v)dvdx
\\
&+\|\chi\|_{L^\infty}H(F^{in}_{\eps}|M)\,.
\ea
\ee
since the relative entropy estimate c) in Theorem \ref{T-ThmMisch} is also satisfied by the approximate solution $F_{\eps,m}$.

Since $\xi$ is compactly supported, putting together (\ref{ConsMom-xi}), (\ref{Young<}), and letting $m\to+\infty$ leads to estimate d). 
\end{proof}

\subsection{The Euler limit}

Let $u^{in}\in\cH(\Om)$, and pick initial data $F^{in}_\eps$ for the Boltzmann equation satisfying
\be\lb{BoltzCondinEul}
\frac1{\eps^2}H(F^{in}_\eps|\cM_{1,\eps u^{in},1})\to 0\quad
\ee

\begin{Thm}\lb{T-Eul<B}
For each $\eps>0$, let $F_\eps$ be a renormalized solution of (\ref{BoltzIEul})-(\ref{AccomEul})-(\ref{BoltzCondin}) satisfying the local momentum
conservation law (\ref{ConsMom}) for each $w\in C^1_c(\bR_+\times\overline\Om)$ satisfying $\Div_xw=0$ and $w\cdot n_x\rstr_{\d\Om}=0$.
Assume that the accommodation parameter $\a$ in the accommodation condition (\ref{AccomEul}) at the boundary depends on the scaling parameter
$\eps$ in such a way that
$$
\a(\eps)=o(\eps)\hbox{ as }\eps\to 0\,.
$$
Then, for each compact $K\subset\overline\Om$, the family
$$
\frac1\eps\int_{\bR^N}vF_\eps dv
$$
is relatively compact in $L^\infty(\bR_+;L^1(K))$, and each of its limit points as $\eps\to 0$ is a dissipative solution of the Euler equations (\ref{IncEul}).
\end{Thm}

\smallskip
Assume that $u^{in}\in\cH(\Om)$ is smooth enough so that the initial-boundary value problem for the Euler equations (\ref{IncEul}) has a classical 
solution on some finite time interval $[0,T]$ --- for instance $u^{in}\in H^s(\Om)$ with $s>\tfrac{N}2+1$, or $u^{in}\in C^{1,\th}$ with $0<\th<1$. In 
that case, the convergence result above can be strengthened with the notion of entropic convergence, invented by Dave Levermore specifically
to handle such problems.

\begin{Def}[C. Bardos, F. Golse, C.D. Levermore \cite{BGL2}]
A family $g_\eps\equiv g_\eps(x,v)$ of $L^1_{loc}(\Om\times\bR^N;Mdxdv)$ is said to converge entropically at order $\eps$ to $g\equiv g(x,v)$ as 
$\eps\to 0$ if the following conditions hold

\smallskip
\noindent
(i) $1+\eps g_\eps\ge 0$ a.e. on $\Om\times\bR^N$ for each $\eps$,

\noindent
(ii) $g_\eps\to g$ weakly in $L^1_{loc}(\Om\times\bR^N;Mdxdv)$ as $\eps\to 0$,

\noindent
(iii) and
$$
\frac1{\eps^2}H(M(1+\eps g_\eps)|M)\to\tfrac12\iint_{\Om\times\bR^N}g(x,v)^2M(v)dxdv
$$
as $\eps\to 0$.
\end{Def}

We recall that, if $g_\eps\to g$ at order $\eps$, then, for each compact $K\subset\Om$, one has
$$
\int_K\int_{\bR^N}(1+|v|^2)|g_\eps(x,v)-g(x,v)|M(v)dvdx\to 0\quad\hbox{ as }\eps\to 0\,.
$$
In other words, entropic convergence implies strong $L^1$ convergence with the weight $(1+|v|^2)M(v)$ (see Proposition 4.11 in \cite{BGL2}.)

\smallskip
Whenever the incompressible Euler equations (\ref{IncEul}) have a classical solution $u$ on $[0,T]\times\Om$, using the weak-strong uniqueness
property of dissipative solutions and the conservation of energy satisfied by classical solutions of (\ref{IncEul}), we arrive at the following stronger
convergence result, which is a straightforward consequence of Theorem \ref{T-Eul<B}. The interested reader is referred to the proof of Theorem 6.2
in \cite{BGL2} where the squeezing argument leading from weak compactness to entropic convergence is explained in detail.

\begin{Cor}
Consider a family $F^{in}_\eps\equiv F^{in}_\eps(x,v)\ge 0$ a.e. of measurable functions on $\Om\times\bR^N$ such that
$$
\frac{F^{in}_\eps(x,v)-M(v)}{\eps M(v)}\to u^{in}(x)\cdot v
$$
entropically of order $\eps$ as $\eps\to 0$, where $u^{in}\in\cH(\Om)$ is smooth enough so that the initial-boundary value problem (\ref{IncEul})
has a classical solution $u$ defined on the time interval $[0,T]$ with $T>0$.

For each $\eps>0$, let $F_\eps$ be a renormalized solution of (\ref{BoltzIEul})-(\ref{AccomEul})-(\ref{BoltzCondin}) satisfying the local momentum
conservation law (\ref{ConsMom}) for each $w\in C^1_c(\bR_+\times\overline\Om)$ that satisfies $\Div_xw=0$ and $w\cdot n_x\rstr_{\d\Om}=0$.
Assume that the accommodation parameter $\a$ in the accommodation condition (\ref{AccomEul}) at the boundary depends on the scaling parameter
$\eps$ in such a way that
$$
\a(\eps)=o(\eps)\hbox{ as }\eps\to 0\,.
$$

Then
$$
\frac{F_\eps(t,x,v)-M(v)}{\eps M(v)}\to u(t,x)\cdot v
$$
entropically of order $\eps$ as $\eps\to 0$, for a.e. $t\in[0,T]$.
\end{Cor}

The proof of Theorem \ref{T-Eul<B} above occupies the remaining part of the present section.

\subsection{The relative entropy inequality}

Statement c) in Theorem \ref{T-ThmMisch} bears on the evolution of the relative entropy of the distribution $F_\eps$ with respect to the uniform
Maxwellian $M=\cM_{1,0,1}$. In the next proposition, we consider the evolution of the relative entropy of the distribution $F_\eps$ with respect to
a local Maxwellian of the form $\cM_{1,\eps w,1}$, where $w$ is a solenoidal velocity field on $\Om$ that is tangential to $\d\Om$.

\begin{Prop}
Let $w\in C^1_c(\bR_+\times\overline\Om)$ be such that
$$
\Div_xw=0\quad\hbox{ and }w\cdot n\rstr_{\d\Om}=0\,.
$$
Then, for each $\eps>0$, renormalized solution $F_\eps$ of the initial-boundary value problem (\ref{BoltzIEul})-(\ref{AccomEul})-(\ref{BoltzCondin})
satisfying the momentum conservation identity (\ref{ConsMom}), also satisy the relative entropy inequality
\be\lb{RelEntrIneq}
\ba
\frac1{\eps^2}H&(F_\eps|\cM_{1,\eps w,1})(t)-\frac1{\eps^2}H(F^{in}_\eps|\cM_{1,\eps w(0,\cdot),1})
\\
&\le-\frac1{\eps^{4+q}}\int_0^t\int_\Om\cP_\eps(s,x)dxds-\frac{\a}{\eps^3}\int_0^t\int_{\d\Om}\cD\cG_\eps(s,x)dxds
\\
&-\frac1{\eps^2}\int_0^t\iint_{\Om\times\bR^N}(v-\eps w(s,x))^{\otimes 2}:\grad_xw(s,x)F_\eps(s,x,v)dxdvds
\\
&-\frac1\eps\int_0^t\iint_{\Om\times\bR^N}(v-\eps w(s,x))\cdot E(w)(s,x)F_\eps(s,x,v)dvdxds
\\
&+\frac{\a}{\eps^2}\int_0^t\int_{\d\Om\times\bR^N}(w(s,x)\cdot v)(v\cdot n_x)_+F_\eps(s,x,v)dvdS(x)ds
\ea
\ee
for each $t>0$.
\end{Prop}

\begin{proof}
We begin with the straightforward identity
$$
\ba
H(F_\eps|\cM_{1,\eps w,1})=H(F_\eps|M)+\iint_{\Om\times\bR^N}F_\eps\ln\left(\frac{M}{\cM_{1,\eps w,1}}\right)dxdv
\\
=H(F_\eps|M)+\iint_{\Om\times\bR^N}F_\eps(\tfrac12|v-\eps w|^2-\tfrac12|v|^2)dxdv
\\
=H(F_\eps|M)+\iint_{\Om\times\bR^N}F_\eps(\tfrac12\eps^2|w|^2-\eps v\cdot w)dxdv\,.
\ea
$$
Thus
\be\lb{EntrId}
\ba
H(F_\eps|\cM_{1,\eps w,1})(t)-H(F_\eps|\cM_{1,\eps w,1})(0)=H(F_\eps|M)(t)-H(F_\eps|M)(0)
\\
+\tfrac12\eps^2\iint_{\Om\times\bR^N}F_\eps(t,x,v)|w(t,x)|^2dxdv
\\
-\tfrac12\eps^2\iint_{\Om\times\bR^N}F^{in}_\eps(x,v)|w(0,x)|^2dxdv
\\
-\eps\iint_{\Om\times\bR^N}F_\eps(t,x,v)v\cdot w(t,x)dxdv
\\
+\eps\iint_{\Om\times\bR^N}F^{in}_\eps(x,v)v\cdot w(0,x)dxdv\,.
\ea
\ee
According to the continuity equation  in statement a) in Theorem \ref{T-ThmMisch}
\be\lb{ConsMass}
\ba
+\tfrac12\eps^2\iint_{\Om\times\bR^N}F_\eps(t,x,v)|w(t,x)|^2dxdv
\\
-\tfrac12\eps^2\iint_{\Om\times\bR^N}F^{in}_\eps(x,v)|w(0,x)|^2dxdv
\\
=\int_0^t\iint_{\Om\times\bR^N}F_\eps(\eps^2\d_t+\eps v\cdot\grad_x)\tfrac12|w|^2dxdvds
\\
=\int_0^t\iint_{\Om\times\bR^N}F_\eps w\cdot(\eps^2\d_tw+\eps v\cdot\grad_xw)dxdvds\,.
\ea
\ee
In (\ref{EntrId}), we replace the term $H(F_\eps|M)(t)-H(F_\eps|M)(0)$ with the right hand side of the inequality of statement c) of Theorem 
\ref{T-ThmMisch}, and use (\ref{ConsMass}) together with (\ref{ConsMom}) to arrive at
\be\lb{RelEntrIneq1}
\ba
H(F_\eps&|\cM_{1,\eps w,1})(t)-H(F_\eps|\cM_{1,\eps w,1})(0)
\\
&\le-\frac1{\eps^{2+q}}\int_0^t\int_\Om\cP_\eps(s,x)dxds-\frac\a\eps\int_0^t\int_{\d\Om}\cD\cG_\eps(s,x)dxds
\\
&+\int_0^t\iint_{\Om\times\bR^N}w\cdot(\eps^2\d_tw+\eps v\cdot\grad_xw)(s,x)F_\eps(s,x,v)dxdvds
\\
&-\int_0^t\iint_{\Om\times\bR^N}(\eps v\cdot\d_tw(s,x)+v^{\otimes 2}:\grad_xw(s,x))F_\eps(s,x,v)dvdxds
\\
&+\a\int_0^t\int_{\d\Om\times\bR^N}(w(s,x)\cdot v)(v\cdot n_x)_+F_\eps(s,x,v)dvdS(x)ds\,.
\ea
\ee
Next we express $\d_tw$ in terms of $E(w)=\d_tw+w\cdot\grad_xw$ and $\grad_xw$: thus
$$
\ba
(\eps v\cdot\d_tw(s,x)+v^{\otimes 2}:\grad_xw(s,x))-w\cdot(\eps^2\d_tw+\eps v\cdot\grad_xw)&
\\
=(v-\eps w)^{\otimes 2}:\grad_xw+\eps(v-\eps w)\cdot E(w)&\,.
\ea
$$
In the right hand side of (\ref{RelEntrIneq1}), we substitute
$$
\ba
\int_0^t\iint_{\Om\times\bR^N}w\cdot(\eps^2\d_tw+\eps v\cdot\grad_xw)(s,x)F_\eps(s,x,v)dxdvds
\\
-\int_0^t\iint_{\Om\times\bR^N}(\eps v\cdot\d_tw(s,x)+v^{\otimes 2}:\grad_xw(s,x))F_\eps(s,x,v)dvdxds
\\
=-\int_0^t\iint_{\Om\times\bR^N}(v-\eps w(s,x))^{\otimes 2}:\grad_xw(s,x)F_\eps(s,x,v)dxdvds
\\
\eps\int_0^t\iint_{\Om\times\bR^N}(v-\eps w(s,x))\cdot E(w)(s,x)F_\eps(s,x,v)dvdxds
\ea
$$
and arrive at the relative entropy inequality (\ref{RelEntrIneq}).
\end{proof}

\subsection{Control of the boundary term}

The relative entropy  inequality (\ref{RelEntrIneq}) is the same as in the one considered in \cite{BGP}, \cite{LionsMas2} and \cite{LSR2}, except
for the boundary term --- i.e. the last term on the right hand side, which is in general not nonpositive. Since the effect of the boundary is our main
interest in this paper, and the only difference with the case of the Cauchy problem treated in \cite{LSR2}, the core of our argument is to obtain a
control of that term.

\begin{Lem}\lb{L-BoundIneq}
With the notations of Theorem \ref{T-ThmMisch}, for each $\eps>0$ and each $w\in C^1_c(\bR_+\times\overline\Om)$ that is tangential on the
boundary $\d\Om$, each renormalized solution $F_\eps$ of the initial-boundary value problem (\ref{BoltzIEul})-(\ref{AccomEul})-(\ref{BoltzCondin}) satisfies the inequality
\be\lb{BoundIneq}
\ba
\frac{\a}{\eps^2}&\int_{\bR^N}(w(s,x)\cdot v)(v\cdot n_x)_+F_\eps(s,x,v)dv
\\
&=
\frac{\a}{\eps}\int_{\bR^N}(w(s,x)\cdot v)(v\cdot n_x)_+g_\eps(s,x,v)M(v)dv
\\
&\le\frac{\a}{2\eps^3}\cD\cG_\eps(s,x)+\frac{\a}{\eps}C(w)\indc_{w(s,x)\not=0}\L_x(F_\eps)(s,x)
\ea
\ee
a.e. in $(s,x)\in\bR_+\times\d\Om$, where
$$
C(w):=\tfrac12\int_{\bR^N}(e^{2\|w\|_{L^\infty}|v|}-2\|w\|_{L^\infty}|v|-1)(v_1)_+M(v)dv\,.
$$
\end{Lem}

We use Young's inequality for a translate of the function $h$ defined in Theorem \ref{T-ThmMisch}, much in the same way as in the proof of
Theorem 6.2 in \cite{BGL2} (see especially pp. 738--739 there).

\begin{proof}
Let $z_0>-1$; for each $z>-1$, set
$$
l(z-z_0):=h(z)-h(z_0)-h'(z_0)(z-z_0)\,.
$$
We recall that the Legendre dual of the function $h$ defined in Theorem \ref{T-ThmMisch} is
$$
h^*(p):=\sup_{z>-1}(pz-h(z))=e^p-p-1\,,\quad p\in\bR\,.
$$
A straightforward computation shows that
$$
\ba
l^*(p)=&\sup_{z>-1}(p(z-z_0)-l(z-z_0))
\\
=&\sup_{z>-1}(p(z-z_0)-h(z)+h(z_0)+h'(z_0)(z-z_0))
\\
=&h(z_0)-(h'(z_0)+p)z_0+\sup_{z>-1}((h'(z_0)+p)z-h(z))
\\
=&h(z_0)-(h'(z_0)+p)z_0+h^*(h'(z_0)+p)
\\
=&h(z_0)-(h'(z_0)+p)z_0+e^{h'(z_0)}e^p-h'(z_0)-p-1
\\
=&(1+z_0)\ln(1+z_0)-z_0-(\ln(1+z_0)+p)z_0
\\
&+(1+z_0)e^p-\ln(1+z_0)-p-1
\\
=&(1+z_0)(e^p-p-1)=(1+z_0)h^*(p)\,.
\ea
$$

Writing $F_\eps=M(1+\eps g_\eps)$ and observing that
\be\lb{L(w.v)}
\int_{\bR^N}(w(s,x)\cdot v)(v\cdot n_x)_+M(v)dv=0\,,\quad (s,x)\in\bR_+\times\d\Om
$$
since $w$ is tangential on $\d\Om$, one has
$$
\ba
\int_{\bR^N}(w(s,x)\cdot v)&(v\cdot n_x)_+F_\eps(s,x,v)dv
\\
=\eps&\int_{\bR^N}(w(s,x)\cdot v)(v\cdot n_x)_+g_\eps(s,x,v)M(v)dv\,.
\ea
\ee
Then
$$
\ba
\int_{\bR^N}&(w(s,x)\cdot v)(v\cdot n_x)_+g_\eps(s,x,v)M(v)dv
\\
&=
\int_{\bR^N}(w(s,x)\cdot v)(v\cdot n_x)_+(g_\eps(s,x,v)-\L_x(g_\eps))M(v)dv
\ea
$$
by (\ref{L(w.v)}) since $\L_x(g_\eps)$ is independent of $v$. 

By definition Young's inequality (or equivalently, by definition of the Legendre dual of $l$), 
$$
2\eps^2(g_\eps-\L_x(g_\eps))(w\cdot v)\le l(\eps(g_\eps-\L_x(g_\eps)))+l^*(\eps(2w\cdot v))\,,
$$
so that
$$
\ba
2\int_{\bR^N}(w\cdot v)(v\cdot n_x)_+g_\eps Mdv
\le
\frac1{\eps^2}\int_{\bR^N}l(\eps(g_\eps-\L_x(\eps)))(v\cdot n_x)_+Mdv&
\\
+
\frac1{\eps^2}\int_{\bR^N}l^*(\eps(2w\cdot v))(v\cdot n_x)_+Mdv&\,.
\ea
$$
First, since $\L_x$ is the average under a probability measure, one has
$$
\ba
\int_{\bR^N}&l(\eps(g_\eps-\L_x(\eps)))(v\cdot n_x)_+Mdv
\\
&=\tfrac1{\sqrt{2\pi}}\L_x(l(\eps(g_\eps-\L_x(\eps))))
\\
&=\tfrac1{\sqrt{2\pi}}\L_x(h(\eps g_\eps)-h(\eps\L_x(g_\eps))-h'(\eps\L_x(g_\eps))(\eps g_\eps-\eps\L_x(g_\eps)))
\\
&=\tfrac1{\sqrt{2\pi}}\L_x(h(\eps g_\eps)-h(\eps\L_x(g_\eps))-h'(\eps\L_x(g_\eps))\tfrac1{\sqrt{2\pi}}\L_x(\eps g_\eps-\eps\L_x(g_\eps)))
\\
&=\tfrac1{\sqrt{2\pi}}(\L_x(h(\eps g_\eps)-h(\eps\L_x(g_\eps)))=\cD\cG_\eps
\ea
$$
On the other hand
$$
\ba
\frac1{\eps^2}&\int_{\bR^N}l^*(\eps(2w\cdot v))(v\cdot n_x)_+Mdv
\\
&=(1+\eps\L(g_\eps))\int_{\bR^N}\frac{e^{2\eps |w||v|}-2\eps |w||v|-1}{\eps^2}(v\cdot n_x)_+Mdv
\\
&\le\L(F_\eps)\int_{\bR^N}(e^{2|w||v|}-2|w||v|-1)(v\cdot n_x)_+Mdv
\ea
$$
--- since, for each $a>0$, the map
$$
\eps\mapsto\frac{e^{a\eps}-a\eps-1}{\eps^2}=a^2\sum_{n\ge 2}\frac{(a\eps)^{n-2}}{n!}
$$
is increasing.

Finally
$$
\ba
\frac{\a}{\eps^2}\int_{\bR^N}(w(s,x)\cdot v)(v\cdot n_x)_+F_\eps(s,x,v)dv
\\
\le
\frac{\a}{\eps}\int_{\bR^N}(w(s,x)\cdot v)(v\cdot n_x)_+g_\eps(s,x,v)M(v)dv
\\
\le\frac{\a}{2\eps^3}\cD\cG_\eps(s,x)+\frac{\a}{\eps}C(w)\indc_{w(s,x)\not=0}\L(F_\eps)(s,x)
\ea
$$
where
$$
C(w):=\tfrac12\int_{\bR^N}(e^{2\|w\|_{L^\infty}|v|}-2\|w\|_{L^\infty}|v|-1)(v_1)_+M(v)dv\,.
$$
\end{proof}

\smallskip
After integrating in $(s,x)$ both sides of (\ref{BoundIneq}), the first term on the right hand side of (\ref{BoundIneq}) will be absorbed by the 
Darrozes-Guiraud information on the right hand side of (\ref{RelEntrIneq}), so that, with Lemma \ref{L-BoundIneq}, the inequality (\ref{RelEntrIneq}) 
is transformed into
\be\lb{RelEntrIneq2}
\ba
\frac1{\eps^2}H&(F_\eps|\cM_{1,\eps w,1})(t)-\frac1{\eps^2}H(F^{in}_\eps|\cM_{1,\eps w(0,\cdot),1})
\\
&\le-\frac1{\eps^{4+q}}\int_0^t\int_\Om\cP_\eps(s,x)dxds-\frac{\a}{2\eps^3}\int_0^t\int_{\d\Om}\cD\cG_\eps(s,x)dxds
\\
&-\frac1{\eps^2}\int_0^t\iint_{\Om\times\bR^N}(v-\eps w(s,x))^{\otimes 2}:\grad_xw(s,x)F_\eps(s,x,v)dxdvds
\\
&-\frac1\eps\int_0^t\iint_{\Om\times\bR^N}(v-\eps w(s,x))\cdot E(w)(s,x)F_\eps(s,x,v)dvdxds
\\
&+\frac{\a}{\eps}C(w)\int_0^t\int_{\d\Om\cap\Supp(w)}\L_x(F_\eps)(s,x)dS(x)ds\,,
\ea
\ee
for all $t\ge 0$.

\subsection{Control of the outgoing mass flux}

In the lemma below, we shall prove that the outgoing mass flux $\L_x(F_\eps)$ is uniformly bounded in $L^1_{loc}(\bR_+\times\d\Om)$, so 
that the last term on the right hand side of (\ref{RelEntrIneq2}) vanishes under the assumption $\a(\eps)=o(\eps)$ as $\eps\to 0$.

\begin{Lem}
With the notations of Theorem \ref{T-ThmMisch}, for each $\eps>0$, each renormalized solution $F_\eps$ of the initial-boundary value problem 
(\ref{BoltzIEul})-(\ref{AccomEul})-(\ref{BoltzCondin}) satisfies the inequality
\be\lb{OutMassFlux}
\ba
\int_{\bR^N}F_\eps(t,x,v)(v\cdot n_x)_+dv\le&\frac1{h(\eta)}\cD\cG_\eps(t,x)&
\\
&+\frac{1}{\sqrt{2\pi}(1-\eta)}\int_{\bR^N}F_\eps(v\cdot n_x)^2dv
\ea
\ee
for all $\eta\in(0,1)$, a.e. in $(t,x)\in\bR_+\times\d\Om$.
\end{Lem}

\begin{proof}
First we recast the Darrozes-Guiraud information in the form 
\be\lb{DGbis}
\ba
\cD\cG_\eps&=\tfrac1{\sqrt{2\pi}}\L_x(h(\eps g_\eps)-h(\eps\L_x(g_\eps)))
\\
&=\tfrac1{\sqrt{2\pi}}\L_x(G_\eps\ln G_\eps-G_\eps-\L_x(G_\eps)\ln\L_x(G_\eps)+\L_x(G_\eps))
\\
&=\tfrac1{\sqrt{2\pi}}\L_x\left(G_\eps\ln\left(\frac{G_\eps}{\L_x(G_\eps)}\right)-G_\eps+\L_x(G_\eps)\right)\,,
\ea
\ee
since
$$
\L_x((G_\eps-\L_x(G_\eps))\ln\L_x(G_\eps))=\L_x((G_\eps-\L_x(G_\eps)))\ln\L_x(G_\eps)=0\,.
$$

Then we consider the integral
$$
\ba
I:&=\L_xG_\eps\int_{\bR^N}(v\cdot n_x)_+^2\wedge 1Mdv
\\
&=\L_xG_\eps\int_{\bR^N}\indc_{|G_\eps/\L_xG_\eps -1|>\eta}(v\cdot n_x)_+^2\wedge 1Mdv
\\
&+\L_xG_\eps\int_{\bR^N}\indc_{|G_\eps/\L_xG_\eps -1|\le\eta}(v\cdot n_x)_+^2\wedge 1Mdv
=:I_1+I_2
\ea
$$
avec $\eta\in]0,1[$.

The first term is estimated in terms of the Darrozes-Guiraud information on the boundary, in view of (\ref{DGbis}):
$$
\ba
I_1&\le\frac1{h(\eta)}\L_xG_\eps\int_{\bR^N}\indc_{|G_\eps/\L_xG_\eps -1|>\eta}
	h\left(\left|\frac{G_\eps}{\L_xG_\eps}-1\right|\right)(v\cdot n_x)_+^2\wedge 1Mdv
\\
&\le
\frac1{h(\eta)}\L_xG_\eps\int_{\bR^N}h\left(\frac{G_\eps}{\L_xG_\eps}-1\right)(v\cdot n_x)_+^2\wedge 1Mdv
\\
&\le
\frac1{h(\eta)}\int_{\bR^N}\left(G_\eps\ln\left(\frac{G_\eps}{\L_xG_\eps}\right)-G_\eps+\L_xG_\eps\right)(v\cdot n_x)_+Mdv
\\
&\le\frac{\sqrt{2\pi}}{h(\eta)}\cD\cG_\eps(t,x)\,.
\ea
$$
As for the second term, one has
$$
\ba
I_2&=\L_xG_\eps\int_{\bR^N}\indc_{|G_\eps/\L_xG_\eps -1|\le\eta}(v\cdot n_x)_+^2\wedge 1Mdv
\\
&\le\frac1{1-\eta}\int_{\bR^N}G_\eps\indc_{|G_\eps/\L_xG_\eps -1|\le\eta}(v\cdot n_x)_+^2\wedge 1Mdv
\\
&\le\frac1{1-\eta}\int_{\bR^N}G_\eps(v\cdot n_x)^2Mdv\,.
\ea
$$
Putting together both estimates gives
$$
I=I_1+I_2\le\frac{\sqrt{2\pi}}{h(\eta)}\cD\cG_\eps(t,x)+\frac1{1-\eta}\int_{\bR^N}F_\eps(v\cdot n_x)^2dv\,.
$$
Since 
$$
\int_{\bR^N}(v\cdot n_x)_+^2\wedge 1Mdv=\int_0^1 z^2e^{-z^2/2}\frac{dz}{\sqrt{2\pi}}=:J>0
$$
is independent of $x$, we conclude that
$$
\int_{\bR^N}F_\eps(t,x,v)(v\cdot n_x)_+dv=\tfrac1{\sqrt{2\pi}}\L_x(G_\eps)(t,x)=\tfrac1{\sqrt{2\pi}}\frac{I}{J}
$$
which, together with the previous inequality, leads to the announced estimate.
\end{proof}

\subsection{Convergence to the incompressible Euler equations}

At this point bring together the relative entropy inequality (\ref{RelEntrIneq}) and the boundary control (\ref{BoundIneq}), thereby arriving
at the estimate
$$
\ba
\frac1{\eps^2}H&(F_\eps|\cM_{1,\eps w,1})(t)-\frac1{\eps^2}H(F^{in}_\eps|\cM_{1,\eps w(0,\cdot),1})
\\
&\le-\frac1{\eps^{4+q}}\int_0^t\int_\Om\cP_\eps(s,x)dxds-\frac{\a}{2\eps^3}\int_0^t\int_{\d\Om}\cD\cG_\eps(s,x)dxds
\\
&-\frac1{\eps^2}\int_0^t\iint_{\Om\times\bR^N}(v-\eps w(s,x))^{\otimes 2}:\grad_xw(s,x)F_\eps(s,x,v)dxdvds
\\
&-\frac1\eps\int_0^t\iint_{\Om\times\bR^N}(v-\eps w(s,x))\cdot E(w)(s,x)F_\eps(s,x,v)dvdxds
\\
&+\frac{\a}{\eps}C(w)\int_0^t\int_{\d\Om\cap\Supp(w)}\L_x(F_\eps)(s,x)dS(x)ds\,.
\ea
$$
Next, we use the pointwise inequality (\ref{OutMassFlux}) with, say, $\eta=\tfrac12$, to control the last integral on the right-hand side above:
\be\lb{RelEntrIneqFinal}
\ba
\frac1{\eps^2}H&(F_\eps|\cM_{1,\eps w,1})(t)-\frac1{\eps^2}H(F^{in}_\eps|\cM_{1,\eps w(0,\cdot),1})
\\
&\le-\frac1{\eps^{4+q}}\int_0^t\int_\Om\cP_\eps(s,x)dxds
\\
&-\frac{\a}{2\eps^3}(1-\tfrac{2\sqrt{2\pi}}{h(1/2)}C(w)\eps)\int_0^t\int_{\d\Om}\cD\cG_\eps(s,x)dxds
\\
&-\frac1{\eps^2}\int_0^t\iint_{\Om\times\bR^N}(v-\eps w(s,x))^{\otimes 2}:\grad_xw(s,x)F_\eps(s,x,v)dxdvds
\\
&-\frac1\eps\int_0^t\iint_{\Om\times\bR^N}(v-\eps w(s,x))\cdot E(w)(s,x)F_\eps(s,x,v)dvdxds
\\
&+\frac{2\a}{\eps}C(w)\int_0^t\int_{\d\Om\cap\Supp(w)}\int_{\bR^N}F_\eps(v\cdot n_x)^2dvdS(x)ds\,.
\ea
\ee
Now, statement d) in Theorem \ref{T-ThmMisch} and the scaling assumption on the accommodation parameter, i.e. $\a(\eps)=o(\eps)$, show 
that, for each $T>0$
\be\lb{EstFvn2}
\frac{2\a}{\eps}C(w)\int_0^t\int_{\d\Om\cap\Supp(w)}\int_{\bR^N}F_\eps(v\cdot n_x)^2dvdS(x)ds\to 0
\ee
uniformly in $t\in[0,T]$ as $\eps\to 0$.

With (\ref{EstFvn2}), the relative entropy inequality (\ref{RelEntrIneqFinal}) is precisely of the same form as the inequality stated as Theorem 5
in \cite{LSR2}. One then concludes by the same argument as in \cite{LSR2}.

\section{Conclusion and final remarks}

As recalled above, the convergence of solutions of the boundary value problem for the Navier-Stokes equations in the large Reynolds number regime 
is an open problem, as well as the validity of the Prandlt equation for the boundary layer.  Convergence to solutions of the Euler Equation, for general 
boundary conditions (including the Dirichlet boundary condition)  are proven under only the most stringent regularity assumptions.

On the other hand the onset of von Karman vortex streets and the Kolmogorov hypothesis on turbulence based on a non zero energy dissipation in 
the large Reynolds (see chapter 5 in \cite{ Frisch}) suggests that, in general, the limit is not a solution of the Euler equation. This is in agreement with 
Kato's criterion \cite{KatoBL} relating the convergence of the solutions of the Navier-Stokes equations with Dirichlet boundary condition to a solution 
of the Euler equations with the vanishing of the viscous energy dissipation in a boundary layer with thickness $O(\Re^{-1})$.

For the Navier slip boundary condition (\ref{SlipFour}) or (\ref{SlipNormDer}), the inviscid limit is established in the present paper is proven under the 
only assumption $\l\to 0$ as $\nu\to 0$. 
 
The following remarks are in order

\begin{itemize}
 
\item at variance with previous results \cite{BeiraoCrispo1,XiaoXin}, no regularity assumption is required for all the results in the present paper, as 
only estimates are used in the proof;

\item what is proved here is the convergence to a dissipative solution (hence to the unique classical whenever it exists); therefore, this convergence 
is also true even if no classical solution exists, or even if the $L^2$ initial data corresponds to a ``wild'' solution \`a la C. DeLellis and L. Szkelyhidi 
(see \cite{deLelSzeke}) for which there is no uniqueness of the Euler solution;

\item therefore the main goal of the present paper is to show the strong similarity between the inviscid limit for the Navier-Stokes equations and the 
fluid dynamic limit for the Boltzmann equation; the accommodation coefficient $\a$, the Mach $\Ma$ and Strou- hal $\Sh$ numbers, the Reynolds
number $\Re$ (see \cite{SoneBk2}, \S 1.9), and the slip coefficient $\l$ are related by
$$ 
\Ma=\Sh=\eps\,,\quad\frac1\Re=o(1)\,,\quad\frac{\alpha}{\epsilon}=\l\,,
$$ 
as $\eps\to 0$, so that the conditions $\a=o(\eps)$ as $\eps\to 0$ and $\l\to 0$ as $\nu\to 0$ are consistent.

\end{itemize}
     
In all case the convergence of the solution of the Navier-Stokes equations to a classical solution of the Euler equations will imples that the energy
dissipation vanishes in the limit, as observed by Kato \cite{KatoBL}. Likewise, the entropic convergence obtained in the present paper in the case 
of renormalized solutions of the Boltzmann equation implies that the sum of the entropy dissipation and of the Darrozes-Guiraud information at the 
boundary vanishes with $\eps$.

By analogy with the work of Kato \cite{KatoBL}, a possible conjecture is that, whenever
$$
\varliminf_{\eps\to 0}\frac{\a(\eps)}{\eps}>0\,,
$$
the vanishing of both the Darrozes-Guiraud information at the boundary and of the entropy production implies that the inviscid fluid dynamic limit
of the Boltzmann equation is described by a solution of the Euler equations.

\smallskip
Other, perhaps less delicate problems could be analyzed by the methods used in the present paper. For instance, Maxwell's accommodation is but
one example in a wide class of nonlocal boundary conditions for the Boltzmann equation; the Navier-Stokes and Euler hydrodynamic limits of the
Boltzmann equation should be considered also for such boundary conditions (see \cite{SoneBk2} \S 1.6). Likewise, the condition (\ref{ConsMom}), 
which may not be verified for all renormalized solutions of the initial boundary value problem for the Boltzmann equation should be removed on 
principle (as in \cite{LSR2}).

In other words, there remain many open problems related to the issues discussed in the present paper, to which we shall return in future publications.



\end{document}